\documentclass[a4paper,final]{siamltex}
\usepackage{array}
\usepackage{amsmath, amssymb,amsfonts}
\usepackage[T1]{fontenc}
\usepackage[latin1]{inputenc}
\usepackage[english,finnish]{babel}
\usepackage{graphicx}
\usepackage{bm}
\numberwithin{equation}{section}
\title{Numerical computations with $\bm{H(\mdiv)}$-finite elements for the Brinkman problem}

\author{Juho Könnö\thanks{Department of Mathematics and Systems Analysis, Aalto University ({\tt jkonno@math.tkk.fi})}
\and Rolf Stenberg\thanks{Department of Mathematics and Systems Analysis, Aalto University}}

\begin{document}

\selectlanguage{english}

\newcommand{\norm}[1]{\|#1\|}
\newcommand{\unorm}[1]{\| #1 \|_{\sigma,t,h}}
\newcommand{\unormloc}[1]{\| #1 \|_{\sigma,t,h,\omega_K}}
\newcommand{\ustarnorm}[1]{\| #1 \|_{\sigma,t,*}}
\newcommand{\pnorm}[1]{|\!|\!| #1 |\!|\!|_{\sigma,t,h}}
\newcommand{\pnormloc}[1]{|\!|\!| #1 |\!|\!|_{\sigma,t,h,\omega_K}}
\newcommand{\unormhalf}[1]{\| #1 \|_{\sigma,t,h/2}}
\newcommand{\pnormhalf}[1]{|\!|\!| #1 |\!|\!|_{\sigma,t,h/2}}
\newcommand{\bb}[1]{\bm{#1}}
\newcommand{\bsigma}{\mathbf{\sigma}}
\newcommand{\beps}{\bm{\varepsilon}}
\newcommand{\bbeta}{\bm{\beta}}
\newcommand{\beeta}{\bm{\eta}}
\newcommand{\pder}[2]{\frac{\partial #1}{\partial #2}}
\newcommand{\mdiv}{\mathrm{div}}
\newcommand{\mrot}{\mathrm{rot}}
\newcommand{\BB}{\mathcal{B}}
\newcommand{\BBp}{\mathcal{B}^*}
\newcommand{\LL}{\mathcal{L}}
\newcommand{\elesum}{\sum_{K \in \mathcal K_h}}
\newcommand{\edgesum}{\sum_{E \in \mathcal E_h}}
\newcommand{\btedgesum}{\sum_{E \in \mathcal E_{h,\bm u_{\bm \tau}}}}
\newcommand{\jump}[1]{[ \! [ #1 ] \! ]}
\newcommand{\eint}[2]{\langle #1, #2 \rangle_E} 
\newcommand{\dkint}[2]{\langle #1, #2 \rangle_{\partial K}} 
\newcommand{\aver}[1]{\{ \! | #1 | \! \}}

\renewcommand{\thefigure}{\arabic{figure}}

\maketitle

\begin{abstract}
The $H(\mdiv)$-conforming approach for the Brinkman equation is studied numerically, verifying the theoretical a priori and a posteriori analysis in~\cite{KSbman,enumath09}. Furthermore, the results are extended to cover a non-constant permeability. A hybridization technique for the problem is presented, complete with a convergence analysis and numerical verification. Finally, the numerical convergence studies are complemented with numerical examples of applications to domain decomposition and adaptive mesh refinement.

\end{abstract}

\section{Introduction}

The Brinkman equation describes the flow of a viscous fluid in a highly porous medium. Mathematically the model is a parameter-dependent combination of the Darcy and Stokes models. For a derivation of and details on the Brinkman equations we refer to~\cite{loi-I,allaire-I,allaire-II,arbogast-etal,Rajagopal07}. Typical applications of the model lie in subsurface flow problems, along with some special applications, such as heat pipes and composite manufacturing~\cite{Goldak07,Griebel10}. The effects of taking the viscosity into account are most pronounced in the presence of large crack or vugs, typical of e.g. real-life oil reservoirs.
The Brinkman model is also used as a coupling layer between a free surface flow and a porous Darcy flow~\cite{EFL09}. Numerical results for the Brinkman flow have been previously presented for the Hsieh-Clough-Tocher element in~\cite{Xu10}, for the Crouzeix-Raviart element in~\cite{BurmanHansbo05}, for the Stokes-based elements with various stabilizations, including interior penalty methods, in~\cite{JSc,dangelo09,peter-mika,Burman08,BurmanHansbo06,BurmanHansbo07}, and for coupling the Stokes and Darcy flows with an SIPG method in~\cite{Kanschat10}. For the $H(\mdiv)$-conforming approximation, numerical results with a subgrid algorithm can be found in~\cite{lazarov10}, and with modified $H(\mdiv)$-elements in~\cite{Xu08,MardalWinther02}.

The structure of the paper is as follows. In Chapter~\ref{sec:model} we briefly recall the mathematical formulation of the model, and introduce the necessary function spaces. Chapter~\ref{sec:fem} carries on to introducing the $H(\mdiv)$-conforming finite element discretization for the problem, along with the Nitsche formulation for assuring conformity and stability in the discrete spaces. We also recall the main results of the a priori and a posteriori analysis carried out in~\cite{KSbman}, along with the postprocessing procedure necessary for the optimal convergence results. The results are extended to cover a non-constant permeability. In Chapter~\ref{sec:hybrid} we introduce a hybridization technique for the parameter dependent problem based on previous hybridization techniques for mixed and DG methods~\cite{egger09,cockburn09,BrezziFortin91}. The practicability of the hybridization and the benefits therein are discussed briefly.

We end the paper with extensive numerical tests in Chapter~\ref{sec:numerics}. We first demonstrate the convergence rates predicted by the theory for both the relative error as well as the a posteriori indicator. Furthermore, the performance of the method is compared with that of a MINI finite element discretization. Next, the importance of the postprocessing method is clarified and convergence of the hybridized method is studied. We also apply the hybridization procedure to domain decomposition. The weak enforcing of the boundary conditions and adaptive refinement techniques are studied in the framework of a flow in a channel with a parameter-dependent boundary condition. The chapter ends with a practical example of the Brinkman flow with actual material parameters from the SPE10 dataset~\cite{spe10} demonstrating the applicability of the estimator to adaptive mesh refinement.

\section{The Brinkman model}\label{sec:model}

Let $\Omega \subset \mathbb{R}^n$, with $n=2,3$, be a domain with a polygonal or polyhedral boundary. We denote by $\bm u$ the velocity field of the fluid and by $p$ the pore pressure. Let $\bm K$ denote the symmetric permeability tensor and $\mu$ and $\tilde{\mu}$ are the dynamic and effective viscosities of the fluid, respectively. With this notation the problem reads~\cite{Rajagopal07,Popov09} 
\begin{align}\label{eq:strong-real}
-\tilde{\mu} \Delta \bm u + \mu \bm K^{-1} \bm u + \nabla p &= \bm f,\quad \mathrm{in}~\Omega, \\
\mdiv~\bm u &= g,\quad \mathrm{in}~\Omega.
\end{align}

To simplify the mathematical analysis we assume the permeability to be of the following diagonal form
\begin{equation}\label{eq:perm_diag}
\bm K^{-1}(\bm x) = \sigma (\bm x)^2 \bm I,\quad \bm x \in \Omega,
\end{equation}
in which $\sigma$ is a strictly positive, piecewise constant function. We furthermore assume both the viscosities to be constant over the whole domain $\Omega$. By scaling the equations with the dynamic viscosity, we arrive at the following scaled problem
\begin{align}\label{eq:strong}
-t^2 \Delta \bm u + \sigma^2 \bm u + \nabla p &= \bm f,\quad \mathrm{in}~\Omega, \\
\mdiv~\bm u &= g,\quad \mathrm{in}~\Omega.
\end{align}
Here the parameter $t$ represents the effective viscosity of the fluid, whereas $\sigma$ reflects the variations in the magnitude of the permeability field. For simplicity, we consider homogenous Dirichlet boundary conditions for the velocity field. For $t > 0$ the boundary conditions are 
\begin{equation}\label{eq:brinkbc}
\bm u = \bm 0.
\end{equation}
For the limiting case $t=0$ we assume the boundary condition
\begin{equation}\label{eq:darcybc}
\bm u \cdotp \bm n = 0.
\end{equation}

For $t > 0$, the equations are formally a Stokes problem. The solution $(\bm u, p)$ is sought in $\bm V \times Q = [H^1_0(\Omega)]^n \times L^2_{0}(\Omega)$. For the case $t=0$ we get the Darcy problem, and accordingly the solution space can be chosen as  $\bm V \times Q = H(\mdiv,\Omega) \times L^2_{0}(\Omega)$ or $\bm V \times Q = [L^2(\Omega)]^n \times [H^1(\Omega) \cap L^2_{0}(\Omega)]$. Here, we focus on the first choice of spaces.

In the following, we denote by $(\cdotp,\cdotp)_D$ the standard $L^2$-inner product over a set $D \subset \mathbb{R}^n $. If $D = \Omega$, the subscript is dropped for convenience. Similarly, $\langle \cdotp,\cdotp \rangle_B$ is the $L^2$-inner product over an $(n-1)$-dimensional subset $B \subset \bar{\Omega}$. We define the following bilinear forms
\begin{align}
a(\bm u,\bm v) &= t^2(\nabla \bm u,\nabla \bm v) + (\sigma^2 \bm u,\bm v),\\ 
b(\bm v,p) &= -(\mdiv~\bm v,p),
\end{align}
and
\begin{equation}
\BB(\bm u,p;\bm v,q) = a(\bm u,\bm v) + b(\bm v,p) + b(\bm u,q).
\end{equation}

The weak formulation of the Brinkman problem then reads: Find $(\bm u,p) \in \bm V \times Q$ such that
\begin{equation}\label{eq:weakprob}
\BB(\bm u,p;\bm v,q) = (\bm f,\bm v) - (g,q),\quad \forall (\bm v,q) \in \bm V \times Q.
\end{equation}

\section{Solution by mixed finite elements}\label{sec:fem}

Let  $\mathcal K_h$ be a shape-regular partition of $\Omega$ into simplices. As usual, the
diameter of an element $K$ is denoted by $h_K$, and the global mesh
size $h$ is defined as $h=\max_{K\in \mathcal K_h}h_K$. We denote by
$\mathcal E_h$ the set of all faces of $\mathcal K_h$. We write $h_E$ 
for the diameter of a face $E$.

We introduce the jump and average of a piecewise smooth scalar
function~$f$ as follows. Let $E = \partial K \cap \partial
K^\prime$ be an interior face shared by two elements $K$ and
$K^\prime$. Then the jump of $f$ over~$E$ is defined by
\begin{equation}
\label{eq:jump} \jump{f}= f|_{K}- f|_{K^\prime},
\end{equation}
and the average as
\begin{equation}
\label{eq:average} \aver{f}= \frac 1 2 (f|_{K} + f|_{K^\prime}).
\end{equation}
For vector valued functions, we define the jumps and averages analogously. Denoting by $\bm n$ the normal vector of a face $E$, we define the tangential component on each face $E$ as
\begin{equation}
\bm u_{\bm \tau} = \bm u - (\bm u \cdotp \bm n)\bm n.
\end{equation}

In addition, we assume that the piecewise constant permeability field agrees with the triangulation $\mathcal K_h$ and that there exist a constant $C \geq 1$ such that 
\begin{equation}
\frac{1}{C} \leq \frac{\sigma^2|_{K}}{\sigma^2|_{K^\prime}} \leq C,\quad \forall K,K^\prime \in \mathcal K_h.
\end{equation}
On each edge $E \in \mathcal E_h$ we define $\bar \sigma^2 = (\sigma^2|_K + \sigma^2|_{K^\prime})/2$. Thus we have $\bar \sigma^2 \sim \sigma^2|_K$ and $\bar \sigma^2 \sim \sigma^2|_{K^\prime}$ for an arbitrary face $E$.

\subsection{The mixed method and the norms}

Mixed finite element discretization of the problem is
based on finite element spaces $\bm V_h\times Q_h\subset
H(\mdiv,\Omega) \times L^2_0(\Omega)$ of piecewise polynomial
functions with respect to~$\mathcal K_h$. We will focus here on the
Raviart-Thomas (RT) and Brezzi-Douglas-Marini (BDM) families of
elements~\cite{BrezziFortin91}. In three dimensions the counterparts are the N\'ed\'elec elements~\cite{Nedelec86}
and the BDDF elements~\cite{bddf87}. That is, for an approximation of
order $k\geq 1$, the flux space~$\bm V_h$ is taken as one of the
following two spaces
\begin{align}
\bm V_h^{RT} &= \{ \bm v \in H(\mdiv,\Omega)~|~\bm v|_K \in [P_{k-1}(K)]^n \oplus \bm x \tilde{P}_{k-1}(K)~\forall K \in \mathcal K_h\},\label{def:vspace} \\
\bm V_h^{BDM} &= \{ \bm v \in H(\mdiv,\Omega)~|~\bm v|_K \in [P_{k}(K)]^n~\forall K \in \mathcal K_h\},\label{def:vspace_bdm}
\end{align}
where $\tilde{P}_{k-1}(K)$ denotes the homogeneous polynomials of degree $k-1$. The pressure is approximated in the same space for both choices of the velocity space, namely
\begin{equation}
Q_h = \{ q \in L^2_0(\Omega)~|~q|_K \in P_{k-1}(K)~\forall K \in \mathcal K_h \}.\label{def:qspace}
\end{equation}
Notice that $\bm V_h^{RT} \subset \bm V_h^{BDM}$. The combination of spaces satisfies the following equilibrium property:
\begin{equation}\label{eq:eqlprop}
\mdiv~\bm V_h \subset Q_h.
\end{equation}

To assure the conformity and stability of the approximation, we use the an SIPG method~\cite{mika-rolf,nitsche-orig}, also known as Nitsche's method, with a suitably chosen stabilization parameter $\alpha > 0$. We define the following mesh-dependent bilinear form
\begin{equation}\label{def:hbilinear}
\BB_h(\bm u, p;\bm v, q) = a_h(\bm u,\bm v) +  b(\bm v,p) + b(\bm u,q),
\end{equation}
in which
\begin{align}\label{def:ahbilinear}
a_h(\bm u,\bm v) &= (\sigma^2 \bm u, \bm v) + t^2 \left[ \elesum (\nabla \bm u,\nabla \bm v)_K \right.\\
 & \left. + \edgesum \{ \frac{\alpha}{h_E} \eint{\jump{\bm u_{\bm \tau}}}{\jump {\bm v_{\bm \tau}}} - \eint{\aver{\pder{\bm u}{n}}}{\jump{\bm v_{\bm \tau}}} - \eint{\aver{\pder{\bm v}{n}}}{\jump{\bm u_{\bm \tau}}}\} \right].\notag
\end{align}
\noindent
Then the discrete problem is to find $\bm u_h \in \bm V_h$ and $p_h \in Q_h$ such that
\begin{equation}\label{eq:feproblem}
\BB_h(\bm u_h,p_h;\bm v,q) = (\bm f,\bm v) - (g,q),\quad \forall (\bm v,q) \in \bm V_h \times Q_h.
\end{equation}

We introduce the following mesh-dependent norms for the problem. For the velocity we use
\begin{equation}\label{def:uhnorm}
\unorm{\bm u}^2 = \elesum \sigma^2 \norm{\bm u}_{0,K}^2 + t^2 \left[ \elesum \norm{\nabla \bm u}_{0,K}^2 + \edgesum \frac{1}{h_E}\norm{\jump{\bm u_{\bm \tau}}}_{0,E}^2 \right],
\end{equation}
and for the pressure
\begin{equation}\label{def:phnorm}
\pnorm{p}^2 = \elesum \frac{h^2_K}{\sigma^2 h^2_K + t^2} \norm{\nabla p}_{0,K}^2 + \edgesum \frac{h_E}{\bar \sigma^2 h^2_E + t^2} \norm{\jump{p}}_{0,E}^2.
\end{equation}
Note that both of the norms are also parameter dependent. To show continuity, we use the somewhat stronger norm
\begin{equation}\label{def:uhstarnorm}
\ustarnorm{\bm u}^2 = \unorm{\bm u}^2 + t^2 \edgesum h_E \norm{\aver{\pder{\bm u}{n}}}_{0,E}^2.
\end{equation}
It is easily shown that the norms~\eqref{def:uhnorm} and~\eqref{def:uhstarnorm} are equivalent in $\bm V_h$. We have the result~\cite{stenberg-95b}, with $C_I > 0$ .
\begin{equation}\label{eq:normalapp}
h_E \norm{\pder{\bm v}{n}}_{0,E}^2 \leq C_I \norm{\nabla \bm v}_{0,K}^2,\quad \forall \bm v \in \bm V_h.
\end{equation}

\subsection{A priori analysis}

For the proofs of the following results, see~\cite{KSbman}. First we note that the method is consistent.
\begin{theorem}\label{th:cons}The exact solution $(\bm u,p) \in \bm V \times Q$ satisfies
\begin{equation}
\BB_h(\bm u,p;\bm v,q) = (\bm f,\bm v) - (g,q),\quad \forall (\bm v,q) \in \bm V_h \times Q_h.
\end{equation}
\end{theorem}

In addition, the bilinear form $a_h(\cdotp,\cdotp)$ is coercive in $\bm V_h$ in the mesh-dependent norm~\eqref{def:uhnorm}.
\begin{lemma}
Let $C_I$ be the constant from the inequality~\eqref{eq:normalapp}. For $\alpha > C_I/4$ there exists a positive constant $C$ such that
\begin{equation}\label{res:a-stab}
a_h(\bm v,\bm v) \geq C \unorm{\bm v}^2,\quad \forall \bm v \in \bm V_h.
\end{equation}
\end{lemma}

With a trivial modification of the proof presented in~\cite{KSbman}, we have the discrete Brezzi-Babuska stability condition.
\begin{lemma}\label{lm:bbcond}
There exists a positive constant $C$ independent of the parameters $t$ and $\sigma$ such that
\begin{equation}\label{eq:bbcond}
\sup_{\bm v \in \bm V_h} \frac{b(\bm v,q)}{\unorm{\bm v}} \geq C \pnorm{q},\quad \forall q \in Q_h.
\end{equation}
\end{lemma}

By the above stability results for $a_h(\cdotp,\cdotp)$ and $b(\cdotp,\cdotp)$, the following full stability result holds, see e.g.~\cite{BrezziFortin91}.

\begin{lemma}\label{lm:fullstab}
There is a positive constant $C$ such that
\begin{equation}
\sup_{(\bm v,q) \in \bm V_h \times Q_h} \frac{\BB_h(\bm r,s;\bm v,q)}{\unorm{\bm v} + \pnorm{q}} \geq C (\unorm{\bm r} + \pnorm{s}), \quad \forall (\bm r,s) \in \bm V_h \times Q_h.
\end{equation}
\end{lemma}

In $H(\mdiv)$, a special interpolation operator $\bm R_h: H(\mdiv,\Omega) \bigcap [L^s(\Omega)]^n \rightarrow \bm V_h$ is required, see~\cite{BrezziFortin91}. We denote by $P_h: L^2(\Omega) \rightarrow Q_h$ the $L^2$-projection. The interpolants possess the following properties:
\begin{equation}\label{eq:divinter}
(\mdiv~(\bm v - \bm R_h \bm v),q) = 0,\quad \forall q \in Q_h,
\end{equation}
\begin{equation}\label{eq:intl2prop}
(\mdiv~\bm v,q - P_h q) = 0,\quad \forall \bm v \in \bm V_h,
\end{equation}
and
\begin{equation}\label{eq:comdiaprop}
\mdiv~\bm R_h = P_h \mdiv.
\end{equation}
By stability and consistency we have the following quasioptimal a priori result shown in~\cite{KSbman}.
\begin{theorem}
There is a positive constant $C$ such that
\begin{equation}
\unorm{\bm u - \bm u_h} + \pnorm{P_h p - p_h} \leq C \unorm{\bm u - \bm R_h \bm u}.
\end{equation}
\end{theorem}
This contains a superconvergence result for $\pnorm{p_h - P_h p}$, which implies that the pressure solution can be improved by local postprocessing. Given full regularity, we conclude the section with the following a priori estimate.
\begin{theorem}
Assuming $\bm u \in [H^{k+1}(\Omega)]^n$ or $\bm u \in [H^{k}(\Omega)]^n$ for BDM and RT elements of order $k$, respectively, we have
\begin{equation}\label{eq:apriori}
\unorm{\bm u - \bm u_h} + \pnorm{P_h p - p_h} \leq 
\begin{cases}
C (\sigma h + t)h^{k-1} \norm{\bm u}_{k},\quad\quad\quad \mathrm{for~RT},\\
C (\sigma h + t)h^{k} \norm{\bm u}_{k+1},\quad \mathrm{for~BDM}.
\end{cases}
\end{equation}
\end{theorem}

\subsection{Postprocessing method}

We recall the postprocessing method proposed in~\cite{KSbman} based on the ideas of~\cite{LovadinaStenberg06}. Due to the varying permeability parameter $\sigma$, we modify the method accordingly. We seek the postprocessed pressure in an augmented space $Q^*_h \supset Q_h$, defined as
\begin{equation}
Q^*_h = 
\begin{cases}
\{ q \in L^2_{0}(\Omega)~|~q|_K \in P_{k}(K)~\forall K \in \mathcal K_h \},\quad\quad \mathrm{for~RT},\\
\{ q \in L^2_{0}(\Omega)~|~q|_K \in P_{k+1}(K)~\forall K \in \mathcal K_h \},\quad \mathrm{for~BDM}.
\end{cases}
\end{equation}

\noindent
The postprosessing method is: Find $p_h^* \in Q^*_h$ such that
\begin{align}
P_h p^*_h &= p_h \\
(\nabla p^*_h,\nabla q)_K &= (t^2 \Delta \bm u_h  - \sigma^2 \bm u_h + \bm f,\nabla q)_K,\quad \forall q \in (I - P_h)Q^*_h|_K.
\end{align}

In~\cite{KSbman} the analysis of the postprocessing method is performed by treating it as an integral part of the problem by embedding it into the bilinear form. Note, that this is solely for mathematical purposes, in computations the postprocessing is performed after the solution of the original system elementwise. The modified bilinear form now reads
\begin{equation}\label{def:ppbilin}
\BBp_h(\bm u,p^*;\bm v,q^*) = \BB_h(\bm u,p^*;\bm v,q^*) + \elesum \frac{h_K^2}{\sigma^2 h_K^2+t^2} (\nabla p^* + \sigma^2 \bm u - t^2 \Delta \bm u,\nabla (I-P_h) q^*)_K.
\end{equation}
\noindent
The postprocessed problem is then: Find $(\bm u_h,p_h^*) \in \bm V_h \times Q_h^*$ such that for every pair $(\bm v,q^*) \in \bm V_h \times Q_h^*$ it holds
\begin{equation}\label{eq:pp-problem}
\BBp_h(\bm u_h,p^*_h;\bm v,q^*) = \LL_h(\bm f,P_h g;\bm v,q^*),
\end{equation}
\noindent
in which
\begin{equation}\label{def:loadfun}
\LL_h(\bm f,g;\bm v,q^*) = (\bm f,\bm v) - (g,q^*) + \elesum \frac{h_K^2}{\sigma^2 h_K^2 + t^2}(\bm f,\nabla (I-P_h) q^*)_K.
\end{equation}

Now using exactly the same arguments as in~\cite{KSbman} for the case $\sigma \equiv 1$, we can show that the solutions of the postprocessing procedure and the modified problem~\eqref{eq:pp-problem} agree, and we have the following quasioptimal a priori result.
\begin{theorem}\label{th:ppapriori}
For the postprocessed solution $(\bm u_h,p_h^*)$ it holds
\begin{align}
\unorm{\bm u - \bm u_h} &+ \pnorm{p - p_h^*} \leq C \inf_{q^* \in Q_h^*}\Big \{ \unorm{\bm u - R_h\bm u} + \pnorm{p - q^*}
 \\
 &+ (\elesum \frac{h_K^2}{\sigma^2 h_K^2 + t^2}\norm{\nabla q^* + \sigma^2 \bm R_h \bm u - t^2 \Delta R_h \bm u - \bm f}_{0,K}^2)^{1/2}\Big\}.
 \notag
\end{align}
\end{theorem}
Assuming full regularity, we have the following optimal a priori result for the postprocessed problem.
\begin{theorem}\label{th:apriori-pp}
Let us assume $(\bm u,p) \in [H^{k+1}(\Omega)]^n \times H^{k+2}(\Omega)$ or $\bm u \in [H^{k}(\Omega)]^n \times \times H^{k+1}(\Omega)$ for BDM and RT elements of order $k$, respectively. Then we have for the solution $(\bm u_h,p_h^*)$ of the postprocessed problem~\eqref{eq:pp-problem}
\begin{equation*}\label{eq:apriori_pp}
\unorm{\bm u - \bm u_h} + \pnorm{p - p_h^*} \leq 
C \{(\sigma h + t)h^{k-1} \norm{\bm u}_{k} + \dfrac{h^{k+1}}{\sigma h + t} \norm{p}_{k+1} \},\quad \mathrm{for~RT},
\end{equation*}
and
\begin{equation*}\label{eq:apriori_pp2}
\unorm{\bm u - \bm u_h} + \pnorm{p - p_h^*} \leq 
C \{(\sigma h + t)h^{k} \norm{\bm u}_{k+1} + \dfrac{h^{k+2}}{\sigma h + t} \norm{p}_{k+2} \},\quad \mathrm{for~BDM}.
\end{equation*}
\end{theorem}

\subsection{A posteriori estimates}\label{subsec:aposteriori}

In this section we introduce a residual-based a posteriori estimator for the postprocessed solution. It should be noted that the postprocessing is vital for a properly functioning estimator. We divide the estimator into two distinct parts, one defined over the elements and one over the edges of the mesh. The elementwise and edgewise estimators are defined as
\begin{multline}\label{def:elemind}
\eta_K^2 = \frac{h_K^2}{\sigma^2 h_K^2 + t^2} \norm{-t^2 \Delta \bm u_h + \sigma^2 \bm u_h + \nabla p_h^* - \bm f}_{0,K}^2 \\
+ (t^2 + \sigma^2 h_K^2) \norm{g - P_h g}_{0,K}^2
\end{multline}
\begin{equation}\label{def:edgeind}
\eta_E^2 = \frac{t^2}{h_E}\norm{\jump{\bm u_{h,\bm \tau}}}_{0,E}^2 + \frac{h_E}{\bar \sigma^2 h_E^2 + t^2} \norm{\jump{t^2 \pder{\bm u_h}{\bm n}}}_{0,E}^2 + \frac{h_E}{\bar \sigma^2 h_E^2 + t^2}\norm{\jump{p_h^*}}_{0,E}^2.
\end{equation}
The global estimator is
\begin{equation}\label{def:globalin}
\eta = \left ( \elesum \eta_K^2 + \edgesum \eta_E^2 \right )^{1/2}.
\end{equation}
Note that setting $t=0$ gives the estimator of~\cite{LovadinaStenberg06} for the Darcy problem. In the following, we address the reliability and efficiency of the estimator and show the terms of the estimator to be properly matched to one another. The estimator introduced is both an upper and a lower bound for the actual error as shown by the following results, provided that a saturation assumption holds. The arguments presented in~\cite{KSbman} hold step-by-step for the estimators~\eqref{def:elemind} and~\eqref{def:edgeind} in conjunction with the scaled norms~\eqref{def:uhnorm} and~\eqref{def:phnorm}.

\begin{theorem}[Reliability]\label{th:apostrel}
There exists a constant $C > 0$ independent of $h$ and the parameters $t$ and $\sigma$ such that
\begin{equation}
\unorm{\bm u - \bm u_h} + \pnorm{p - p_h^*} \leq C \eta.
\end{equation}
\end{theorem}

\begin{theorem}[Efficiency]\label{th:aposteff}There exists a constant $C > 0$ independent of $h$, $t$, and $\sigma$ such that 
\begin{align}
\eta^2 \leq C\Big \{&\unorm{\bm u - \bm u_h}^2 + \pnorm{p - p_h^*}^2 
\\
&+ \elesum \Big( \frac{h_K^2}{\sigma^2 h_K^2 + t^2} \norm{\bm f - \bm f_h}_{0,K}^2 + (t^2 + \sigma^2 h_K^2) \norm{g - P_h g}_{0,K}^2 \Big) \Big\}.
\notag
\end{align}
\end{theorem}

Thus for the displacement $\bm u_h$ and the postprocessed pressure $p_h^*$ we have by Theorems~\ref{th:apostrel} and~\ref{th:aposteff} a reliable and efficient indicator  for an elementwise constant permeability parameter $\sigma$ and all values of the effective viscosity parameter $t$. In addition, we have the localized lower bound
\begin{align}
\eta_K^2 + \eta_E^2 \leq C\Big \{&\unormloc{\bm u - \bm u_h}^2 + \pnormloc{p - p_h^*}^2 
\\
&+ \sum_{K \in \omega_K} \Big( \frac{h_K^2}{\sigma^2 h_K^2 + t^2} \norm{\bm f - \bm f_h}_{0,K}^2 + (t^2 + \sigma^2 h_K^2) \norm{g - P_h g}_{0,K}^2 \Big) \Big\},
\notag
\end{align}
in which $\omega_K \subset \Omega $ is the patch of elements surrounding an element $K$, and the subscripted norms above are evaluated only over the elements in $\omega_K$. Thus we have a strong indication of the applicability of the estimator to adaptive refinement.

\section{Hybridization}\label{sec:hybrid}

A well-known method for dealing with the indefinite system resulting from the Darcy equation is the hybridization technique introduced in~\cite{veubeke65,BrezziFortin91}. The idea is to enforce the tangential continuity via Lagrange multipliers chosen suitably and relaxing the continuity requirement on the finite element space. Thus, we drop the normal continuity requirement in the spaces $\bm V_h^{BDM}$ and $\bm V_h^{RT}$ and denote these discontinuous counterparts by $\tilde{\bm V_h}$. In addition, we introduce the corresponding multiplier spaces
\begin{align}
\Lambda_h^{BDM} &= \{ \lambda \in [L^2(\mathcal E_h)]^{n-1}~|~\lambda \in P_k(E),~E \in \mathcal E_h,~\lambda|_E = 0,~E \subset \partial \Omega \},\\
\Lambda_h^{RT} &= \{ \lambda \in [L^2(\mathcal E_h)]^{n-1}~|~\lambda \in P_{k-1}(E),~E \in \mathcal E_h,~\lambda|_E = 0,~E \subset \partial \Omega \},
\end{align}
in which $\mathcal E_h$ denotes the collection of all faces of the mesh. It can be easily shown, that the normal continuity of a discrete flux $\bm u_h \in \tilde{\bm V_h}$ is equivalent to the requirement
\begin{equation}\label{eq:normalcont}
\elesum \dkint{\bm u_h \cdotp \bm n}{\mu} = 0,\quad \forall \mu \in \Lambda_h.
\end{equation}
Accordingly, the original finite element problem~\eqref{eq:feproblem} can be hybridized in the following form: Find $(\bm u_h,p_h,\lambda_h) \in \tilde{\bm V_h} \times Q_h \times \Lambda_h$ such that
\begin{align}
\BB_h(\bm u_h,p_h;\bm v,q) + \elesum \dkint{\bm v \cdotp \bm n}{\lambda_h} &= (\bm f,\bm v) + (g,q),\\
\elesum \dkint{\bm u_h \cdotp \bm n}{\mu} &= 0
\end{align}
for all $(\bm v,q,\mu) \in \tilde{\bm V_h} \times Q_h \times \Lambda_h$. Due to~\eqref{eq:normalcont}, the solution $(\bm u_h,p_h)$ of the hybridized system coincides with that of the original system. Thus, we need not modify the postprocessing procedure even if we drop the continuity requirement from the velocity space.

\subsection{Hybridization of the Nitsche term}\label{sec:nitsche-hyb}

However, now the matrix block corresponding to the bilinear form $\BB_h(\bm u_h,p_h;\bm v,q)$ is a block diagonal system only for the special case $t=0$, and for a non-zero effective viscosity we cannot eliminate the variables locally. To alleviate this problem we introduce a second hybrid variable for the Nitsche term of the velocity, see e.g.~\cite{cockburn09}. Recall, that the velocity-velocity term of the bilinear form $\BB_h(\bm u_h,p_h;\bm v,q)$ is
\begin{align}
a_h(\bm u,\bm v) &= (\sigma^2 \bm u, \bm v) + t^2 \left[ \elesum (\nabla \bm u,\nabla \bm v)_K \right. \\
 &+ \left. \edgesum \{ \frac{\alpha}{h_E} \eint{\jump{\bm u_{\bm \tau}}}{\jump {\bm v_{\bm \tau}}} - \eint{\aver{\pder{\bm u}{n}}}{\jump{\bm v_{\bm \tau}}} - \eint{\aver{\pder{\bm v}{n}}}{\jump{\bm u_{\bm \tau}}}\}\right].\notag
\end{align}
To this end, we follow~\cite{egger09}, and formally introduce the mean value of $\bm u_{\bm \tau}$ as a new variable, $\bm m = \frac 1 2 (\bm u_{1,\bm \tau}  + \bm u_{2, \bm \tau})$. 
Thus we can write the tangential jump as
\begin{equation}\label{eq:lagrangedef}
\jump{\bm u_{\bm \tau}} = 2(\bm u_{1,\bm \tau} - \bm m) = -2(\bm u_{2, \bm \tau} - \bm m).
\end{equation}
Now using the new hybrid variables the bilinear form $a_h(\bm u,\bm v)$ can be rewritten as
\begin{align}\label{def:ah_hybrid}
a_h(\bm u,\bm m;\bm v,\bm r) &= (\sigma^2 \bm u, \bm v) + t^2 \elesum \{ (\nabla \bm u,\nabla \bm v)_K  + \frac{2 \alpha}{h_E} \dkint{\bm u_{\bm \tau} - \bm m}{\bm v_{\bm \tau} - \bm r}\notag\\
 & \quad -  \dkint{\pder{\bm u}{n}}{\bm v_{\bm \tau} - \bm r} - \dkint{\pder{\bm v}{n}}{\bm u_{\bm \tau} - \bm m} \}. \notag
\end{align}
Here, the hybrid variable $\bm m$ belongs to a space $\mathcal M_h \subset [L^2(\mathcal E_h)]^{n}$, the choice of which will be discussed subsequently. In addition, we introduce a slightly modified version of the norm~\eqref{def:uhnorm} to encompass both the velocity and the hybrid variable:

\begin{equation}\label{def:uhnorm_hybrid}
\unorm{(\bm u,\bm m)}^2 = \elesum \sigma^2 \norm{\bm u}_{0,K}^2 + t^2 \left[ \elesum \norm{\nabla \bm u}_{0,K}^2 + \edgesum \frac{1}{h_E}\norm{\bm u_{\bm \tau} - \bm m}_{0,E}^2 \right].
\end{equation}

Since for the exact solution the jumps disappear, the bilinear form is consistent. Using exactly the same arguments as those presented in~\cite{egger09} for~\eqref{res:a-stab}, we have
\begin{lemma}
The hybridized bilinear form $a_h(\cdotp,\cdotp;\cdotp,\cdotp)$ is coercive in the discrete spaces $\bm V_h \times \mathcal M_h$, that is there exists a positive constant $C$ such that
\begin{equation}\label{res:a-stab-hyb}
a_h(\bm v,\bm m;\bm v,\bm m) \geq C \unorm{(\bm v,\bm m)}^2,\quad \forall (\bm v,\bm m) \in \bm V_h \times \mathcal M_h.
\end{equation}
\end{lemma}
Note, that the stability holds for any choice of the space $\mathcal M_h$. For complicated problems, this gives great flexibility. Thus, due to consistency and stability, we get optimal convergence rate as long as the space $\mathcal M_h$ is rich enough. Here we choose
\begin{equation}
\mathcal M_h = \{ \bm m \in [L^2(\mathcal E_h)]^n~|~Q(E) \bm m|_E \in [P_{k}(E)]^{n-1},\quad \forall E \in \mathcal E_h \},
\end{equation}
in which $Q(E)$ is the coordinate transformation matrix from the global $n$-dimensional coordinate system to the local $(n-1)$-dimensional coordinate system on the face $E$. Let $\bm P_h: [L^2(E)]^{n-1} \rightarrow \mathcal M_h$ be the $L^2$ projection on the faces. We then get the following interpolation estimate.
\begin{lemma}
Let $\bm u$ be such that $\bm u|_K \in [H^{s+1}(K)]^n$ for $\frac 1 2 < s \leq k$. Then it holds
\begin{equation}
\unorm{(\bm u - \bm R_h \bm u,\bm u_{\bm \tau} - \bm P_h \bm u_{\bm \tau})} \leq C (\sigma h + t) h^s \norm{\bm u}_{s+1}.
\end{equation}
\end{lemma}
\begin{proof}
We proceed by direct computation. Scaling and the Bramble-Hilbert lemma~\cite{braess} yield
\begin{multline*}
\unorm{(\bm u - \bm R_h \bm u,\bm u_{\bm \tau} - \bm P_h \bm u_{\bm \tau})}^2 \leq \elesum \sigma^2 \norm{\bm u - \bm R_h \bm u}_{0,K}^2 + t^2 \left[\elesum \norm{\nabla(\bm u - \bm R_h \bm u)}_{0,K}^2 \right.\\
\left. +\edgesum \left( \frac{1}{h_E} \norm{(\bm u - \bm R_h \bm u)_{\bm \tau}}_{0,E}^2  + \frac{1}{h_E} \norm{(\bm u_{\bm \tau} - \bm P_h \bm u_{\bm \tau})}_{0,E}^2 \right) \right] \\
\leq C \left( \sigma^2 h^{2s+2} \norm{\bm u}_{s+1}^2 + t^2 \elesum \{ h_K^{2s} \norm{\bm u}_{s+1,K}^2 + h_K^{2s} \norm{\bm u_{\bm \tau}}_{s+1/2,K}^2 \} \right).
\end{multline*}
The result is immediate after taking the square root.
\end{proof}

Combining the interpolation results with the consistency and ellipticity properties yields an optimal convergence rate for the velocity.
\begin{theorem}
Assuming sufficient regularity, for the finite element solution $(\bm u_h,\bm m_h)$ of the hybridized system it holds
\begin{equation}
\unorm{(\bm u - \bm u_h,\bm u_{\bm \tau} - \bm m_h))} \leq C (\sigma h + t) h^s \norm{\bm u}_{s+1}.
\end{equation}
\end{theorem}
Thus the hybridization preserves the convergence rates of Theorem~\ref{th:apriori-pp}.

The residual a posteriori estimator of Section~\ref{subsec:aposteriori} can be modified to handle the hybrid variable by changing the edgewise estimator~\eqref{def:edgeind} to
\begin{equation}\label{def:edgeind_mod} 
\eta_E^2 = \frac{t^2}{h_E}\norm{\bm u_{h,\bm \tau} - \bm m_h}_{0,E}^2 + \frac{h_E}{\bar \sigma^2 h_E^2 + t^2} \norm{\jump{t^2 \pder{\bm u_h}{\bm n}}}_{0,E}^2 + \frac{h_E}{\bar \sigma^2 h_E^2 + t^2}\norm{\jump{p_h^*}}_{0,E}^2.
\end{equation}
Following the lines of the proof in~\cite{KSbman} it is easy to prove that also the modified estimator is both sharp and reliable. This will be demonstrated in numerical experiments in Section~\ref{sec:numerics}.

\subsection{Implementation and local condensation}

In practice, it is beneficial to choose the hybrid variable $\bm m$ slightly differently, namely as the weighted average $\bm m = \frac t 2 (\bm u_{1,\bm \tau}  + \bm u_{2, \bm \tau})$. Now the hybridized bilinear form can be written as

\begin{multline}\label{def:ah_hybrid2}
a_h(\bm u,\bm m;\bm v,\bm r) = (\sigma^2 \bm u, \bm v) + \elesum \frac{2\alpha}{h_E} \dkint{\bm m}{\bm r} \\
+ t \elesum \{ \dkint{\pder{\bm u}{n}}{\bm r} + \dkint{\pder{\bm v}{n}}{\bm m} - \frac{2\alpha}{h_E} \dkint{\bm u_{\bm \tau}}{\bm r} - \frac{2\alpha}{h_E} \dkint{\bm v_{\bm \tau}}{\bm m} \} \notag \\
+ t^2 \elesum \{ (\nabla \bm u,\nabla \bm v)_K  + \frac{2 \alpha}{h_E} \dkint{\bm u_{\bm \tau}}{\bm v_{\bm \tau}} - \dkint{\pder{\bm u}{n}}{\bm v_{\bm \tau}} - \dkint{\pder{\bm u}{n}}{\bm u_{\bm \tau}} \}. \notag
\end{multline}
Note, that now we get a $t$-independent part for the hybrid variable and the system remains solvable even in the limit $t \rightarrow 0$. It is clear that all of the results in Section~\ref{sec:nitsche-hyb} hold also for the scaled hybrid variable.

The main motivation for the hybridization procedure is to break all connections in the original saddlepoint system to allow for local elimination of the velocity and pressure variables at the element level. After hydridization the matrix system gets the following form  where $A$ is the matrix corresponding to the bilinear form $a_h(\cdotp,\cdotp)$, $B$ to $b(\cdotp,\cdotp)$, whilst $C$ and $D$ represent the connecting blocks for the hybrid variables for normal continuity and the Nitsche terms, respectively. $M$ is the mass matrix for the Nitsche hybrid variable.

\begin{equation}\label{eq:hybmatrix}
\begin{pmatrix}
A & B^T & C^T & D^T \\
B & 0 & 0 & 0 \\
C & 0 & 0 & 0 \\
D & 0 & 0 & M
\end{pmatrix}.
\end{equation}

Since $A$ and $B$ are now block diagonal matrices, they can be inverted on the element level and we get the following symmetric and positive definite system for the hybrid variables. We denote the by $H$ the following matrix that can be computed elementwise.
\begin{equation}
H := A^{-1}B^T(BA^{-1}B^T)^{-1}BA^{-1} - A^{-1}.
\end{equation}
The matrix $H$ is positive definite and symmetric. Eliminating the velocity and pressure from the system matrix~\eqref{eq:hybmatrix} yields the following system matrix for the hybrid variables $(\lambda,\bm m)$ corresponding to the normal continuity and tangential jumps, respectively.
\begin{equation}\label{eq:schur_system}
\begin{pmatrix}
CHC^T & CHD^T \\
DHC^T & DHD^T + M
\end{pmatrix}.
\end{equation}
Evidently the resulting system is symmetric and positive definite. Note, that whilst  the connecting block $D$ will vanish as $t \rightarrow 0$, the $M$ block does not depend on $t$, thus the whole system remains invertible even in the limit.

\subsection{Application to domain decomposition}

The hybridized formulation is well-suited to solving large problems with the domain decomposition method. The hybrid variables readily form a discretization for the skeleton of the domain decomposition method for any choice of non-overlapping blocks. The local problems are of the Dirichlet type, and the domain decomposition method can be implemented easily using local solvers on the subdomains.

Furthermore, only the skeleton of the domain decomposition mesh can be hybridized using conventional $H(\mdiv)$-elements for the saddle point system in the subdomains. We also have great flexibility in the choice of the hybridized variables, thus allowing one to use a lower number of degrees of freedom on the skeleton when computational resources are limited or alternatively construct a coarser approximation for use as a preconditioner. Most importantly, the mass conservation property is retained over the subdomain boundaries, thus making the method a viable alternative for multiphase flow computations.

\section{Numerical tests}\label{sec:numerics}

In this section, we shall numerically demonstrate the performance of the method. In all of the convergence tests we are mainly interested in the effect of the ratio between the Stokes and Darcy terms $-t^2\Delta \bm u$ and $\sigma^2 \bm u$. Thus for the convergence tests we reduce the problem to a one-parameter family of problems by choosing $\sigma = 1$ on the whole of the domain $\Omega$. This allows us to visualize different behavior of the numerical method in the Stokes and Darcy regimes for varying parameter values.

This choice of parameterization is made for clarity of the numerical analysis, since it allows the Brinkman problem to be easily interpreted as a singular perturbation of the Darcy problem, on which the finite element scheme proposed is based. More exactly, the underlying Darcy flow problem for which the behaviour of the $H(\mdiv)$-conforming discretization is well-understood is kept constant as regards the ratio of the permeability parameter to the mesh size. The effects of adding the viscous term are then studied by gradually increasing value of the $t$ parameter, which is in accordance with the numerical nature of the problem.

Since the main value of the convergence tests lies in showing the performance of the method throughout the parameter range in an idealized case with a known solution, the permeability can be chosen not to vary spatially without losing generality. Thus, the conversion to a one-parameter family can be interpreted also as a constant scaling of the pressure and the loading in the original model~\eqref{eq:strong} as~\cite{JSc}
\begin{align}\label{eq:strong-scaled}
-\left(\frac{t}{\sigma}\right)^2 \Delta \bm u + \bm u + \frac{1}{\sigma^2}\nabla p &= \frac{1}{\sigma^2}\bm f,\quad \mathrm{in}~\Omega, \\
\mdiv~\bm u &= g,\quad \mathrm{in}~\Omega.
\end{align}
Accordingly, in the physically more realistic case of approximately constant viscosity and varying permeability, the converge results can be interpreted by regarding the above scaling of the pressure as a variation in the magnitude of the driving force exerted by a given pressure gradient.

First, we test the optimal convergence properties of the method along with the performance of the a posteriori estimator in two cases with different regularity properties. The proposed method is also compared with the Stokes-type approach using the MINI element. We proceed to demonstrate the effect and the importance of the postprocessing procedure. Next, the convergence of the hybridized method is studied in the framework of domain decomposition methods. Our fourth test is a flow in a channel, demonstrating the performance of Nitsche's method in assigning the boundary conditions and the applicability of the error indicator to adaptive mesh refinement. We end the section with a realistic example employing permeability data from the SPE10 dataset. In all of the test cases the ratio $h/t$ is the ratio $1/(t\sqrt{N})$, in which $N$ is the number of degrees of freedom. For a uniform mesh we have $h/t \approx 1/(t\sqrt{N})$. Note, that this holds only in the two-dimensional case considered in the computations. The data approximation term in the a posteriori estimator is neglected in the computations. In all but the last test case in which real-life values are used, the parameters  as well as the size of the domain are dimensionless.

\subsection{Convergence tests}

For the purpose of testing the convergence rate, we pick a pressure $p$ such that $\nabla p$ is a harmonic function. Thus, $\bm u = \nabla p$ is a solution of the problem for every $t \geq 0$. In polar coordinates $(r,\Theta)$ the pressure is chosen as
\begin{equation}
p(r,\Theta) = r^\beta \sin(\beta \Theta) + C.
\end{equation}
The constant $C$ is chosen such that the pressure will have a zero mean value. Moreover, we have $p \in H^{1+\beta}(\Omega)$, and subsequently $\bm u \in [H^\beta(\Omega)]^n$, see~\cite{Gekeler06}. In the following, we have tested the convergence with a wide range of different parameter values, and the results are plotted with respect to the ratio of the viscosity parameter $t$ to the mesh size $h$. Our aim is to demonstrate numerically, that the change in the nature of the problem indeed occurs at $t=h$, and that the convergence rates are optimal in both of the limiting cases. First we choose $\beta = 3.1$ to test the convergence rates in a smooth situation. With the first-order BDM element we are expecting $h^2$ converge in the Darcy end of the parameter range and $h$ in the Stokes limit.

As is visible from the results in Figures~\ref{fig:thconv_b31} and~\ref{fig:rates_b31}, the behaviour of the problem changes numerically when $t=h$. Thus, even though in practical applications almost always $t > 0$, numerically the problem behaves like the Darcy problem when $t < h$. As can be seen from Figure~\ref{fig:rates_b31}, the converge rates follow quite closely those given by the theory. Furthermore, both the actual error and the a posteriori indicator behave in a similar manner. Notice, that the convergence rate exhibits a slight dip at the point in which the nature of the problem changes. This is a result of the dominating error component changing from pressure to velocity error as we pass into the Stokes regime.

To show the applicability of the a posteriori indicator to mesh refinement, we consider the more irregular case $\beta = 1.52$. Our refinement strategy consists of refining those elements in which the error exceeds 50 percent of the average value. The treshold is halved until at least five percent of the elements have been refined. The edge estimators are shared evenly between the neighbouring elements. When compared to uniform refinement in Figure~\ref{fig:rates_b152}, Figure~\ref{fig:rates_adaptiveb152} shows that the convergence is considerably faster with adaptive refinement. Furthermore, adaptive refinement seems to alleviate the problem of convergence rate drop at the numerical turning point $t=h$, as demonstrated in Figures~\ref{fig:thconv_b152} and~\ref{fig:thconv_adaptiveb152}. Clearly these results indicate that the a posteriori indicator proposed gives reasonable local bounds and can thus be used in adaptive mesh refinement.

\subsection{Comparison with the MINI element}

A common choice for solving the Stokes problem is the classical MINI element~\cite{arnold-brezzi-fortin-84}. This element has been applied to the Brinkman problem and thoroughly analyzed both theoretically and numerically in~\cite{mika-rolf,JSc}. We use the same test case as above with the regularity parameter set to $\beta = 3.1$. Notice, that the mesh-dependent norms~\eqref{def:uhnorm} and~\eqref{def:phnorm} reduce to the ones presented in~\cite{mika-rolf} when a continuous velocity-pressure pair is inserted. Thus we can use the same mesh-dependent norm for computing the error for both of the elements and the results are comparable with one another.

As can be seen from the results in Figures~\ref{fig:thconv_comparisonb31} and~\ref{fig:rates_comparisonb31}, the convergence rate for the MINI element is as expected of the order $h$ throughout the parameter regime. For the BDM1 element, on the other hand, we get $h^2$ convergence in the Darcy regime, and after a slight dip at the turnaround point convergence relative to $h$. Turning our attention to Figure~\ref{fig:thconv_comparisonb31} reveals that the behaviour of the absolute value of the relative error differs vastly for the two elements. Clearly, the BDM element outperforms the MINI element in the Darcy regime by several decades of magnitude, whereas in the Stokes regime the performance of the MINI element is superior. This implicates that it is impossible to clearly tell which element is superior for the Brinkman flow. However, usually real-life reservoirs consist mostly of porous stone with scattered vugs and cracks. Thus the volume of the Stokes-type regime is often small compared to the Darcy regime, implying that the good performance of the BDM element in its natural operation conditions might offer significant performance increase for the overall simulation. In problems with a greater fraction of void space, such as filters with large free-flow areas separated by permeable thin layers, the Stokes-based elements might be a more natural choice.

\subsection{Postprocessing}

In this section, we show the necessity of the postprocessing by comparing the behaviour of both the exact error and the a posteriori estimator for the original and the postprocessed pressure. We shall use the same test case as in the previous sections. First we choose $\beta = 3.1$ for testing the effect of the postprocessign on the convergence. On the second run we choose $\beta=1.52$ and use the same mesh refinement strategy as before to show the necessity of the post-processing for the usefulness of the error estimator.

From the results of Figure~\ref{fig:thconv_noppb31}, it is immediately evident that in the case of a small parameter $t$ corresponding to a Darcy-type porous flow the postprocessing procedure is of vital importance for the method to work. However, as the viscosity increases the weighting of the pressure error changes, and the norm is more tolerable of variations in the pressure. Once again, this change in behaviour appears exactly at $t=h$. As regards convergence rate, the non-postprocessed method is able to attain a close-to-optimal rate in the Stokes regime, cf. Figure~\ref{fig:rates_noppb31}. The same pattern is evident also with the more irregular test case with $\beta = 1.52$ as shown in Figures~\ref{fig:thconv_noppadab152} and~\ref{fig:rates_noppadab152}. When in the Darcy regime, the indicator simply does not work in adaptive refinement since the pressure solution lacks the necessary extra accuracy. However, when crossing into the Stokes regime, convergence starts to occur, and the adaptive procedure achieves a rather high rate of convergence. Evidently, postprocessing is vital for the method in the Darcy regime. Even though the method seems to work without postprocessing in the Stokes regime, one cannot guarantee convergence and thus the method should always be used only in conjunction with the postprocessing scheme for the pressure. The cost of solving the postprocessed pressure is negligible compared to the total workload since the procedure is performed elementwise. 

Note that with postprocessing using the BDM family of elements is more economical than using the RT family with respect to the number of degrees of freedom, since we can use initially one order lower approximation for the pressure, and still get the same order of polynomial approximation and convergence after the postprocessing procedure. 

\subsection{Hybridized method}

Here we study the convergence of the fully hybridized method for different parameter values using the modified norm~\eqref{def:uhnorm_hybrid} and the a posteriori estimator~\eqref{def:edgeind_mod}. We employ once again the same exact solution as in the other convergence tests with the same values for the regularity parameter $\beta$. The errors are measured using the modified norm~\eqref{def:uhnorm_hybrid}.

First we hybridize all of the edges in $\mathcal E_h$. As Figures~\ref{fig:thconv_hyb31} to~\ref{fig:rates_hyb152} clearly show, the hybridized method behaves in an identical manner compared to the standard formulation, both in the case of a regular and an irregular exact solution. Thus it can be concluded that the error induced by hybridizing the jump terms inexactly is negligible, and the proposed convergence rates are attained.

In the second test we divide the domain $\Omega$ into 16 triangular subdomains, and hybridize the finite element spaces only on the domain boundaries and employ standard $H(div)$-conforming BDM1 elements in the subdomains. Accordingly, the estimator is modified only on the hybridized faces. Note, that since the hybridization for the tangential jumps is not exact, the system of equations for the domain decomposition method differs from the one for the fully hybridized method, since we are using the original bilinear form~\eqref{def:hbilinear} in the subdomains. As is evident from Figures~\ref{fig:thconv_dd31} to~\ref{fig:rates_dd152}, we have the correct convergence for the hybridized domain decomposition method, too.

Lastly, we investigate the condition number of the Schur complement matrix~\eqref{eq:schur_system} for $\beta = 3.1$ and values of $t$ ranging from $t=0$ to $t=10^3$. The total number of degrees of freedom is kept approximately constant, whilst the number of the subdomains in the domain decomposition method is varied. The subdomains are approximately equal in size. As is evident from Figure~\ref{fig:condtest}, in the Darcy regime the condition number is rather insensitive to the value of the parameter $t$, however as $t$ is increased and we pass to the Stokes regime the condition number grows as $t^2$. Furthermore, in the Darcy regime we observe an increasing condition number related to the size of the subdomain $\Omega_i$ as ${\mbox{diam}(\Omega_i)}^{-1}$, as one would expect.

\subsection{Pressure driven flow in a channel}

A viscous flow in a narrow channel driven by a pressure drop gives us a test case for which the exact solution is known. The flow is driven by a linear pressure drop with no-slip boundary conditions for $t>0$. For the Darcy case $t=0$, only the normal component of the velocity vanishes on the boundary. We will test the convergence with Nitsche's method for the tangential boundary condition with adaptive mesh refinement. Denote by $\mathcal E_{h,\bm u_{\bm \tau}}$ the collection of edges residing on the part of the boundary $\Gamma_{\bm u_{\bm \tau}} \subset \partial \Omega$ on which we set the tangential velocity. To set the tangential velocity we then add the term 
\begin{equation}
t^2 \btedgesum \{ \frac{\alpha}{h_E} \eint{\bm u_{\bm \tau}}{\bm v_{\bm \tau}} - \eint{\pder{\bm u}{n}}{\bm v_{\bm \tau}} - \eint{\pder{\bm v}{n}}{\bm u_{\bm \tau}}\}
\end{equation}
to the bilinear form $a_h(\cdotp,\cdotp)$. The loading is augmented by the term
\begin{equation}
t^2 \btedgesum \{ \frac{\alpha}{h_E} \eint{\bm u_{D,\bm \tau}}{\bm v_{\bm \tau}} - \eint{\pder{\bm v}{n}}{\bm u_{D, \bm \tau}}\}.
\end{equation}
in which $\bm u_D$ denotes the velocity boundary condition. Thus, in the limit $t=0$ a standard Darcy problem with Dirichlet normal boundary conditions is solved. Finally, the error estimator~\eqref{def:edgeind} is modified by adding the corresponding boundary face terms to the estimator.

In the unit square we take the pressure to be $p = -x + \frac 1 2$. Then zero boundary conditions for the velocity for $y=0$ and $y=1$ give the exact solution $\bm u = (u, 0)$, with the $x$-directional velocity given by
\begin{equation}
u=
\begin{cases}
(1 + e^{1/t} - e^{(1-y)/t} - e^{y/t})/(1+e^{1/t}),\quad t > 0 \\
1,\quad t=0.
\end{cases}
\end{equation}

As Figures~\ref{fig:poismesh05} through~\ref{fig:poismesh0005} demonstrate, the adaptive process is able to catch the boundary layer effectively in the case where no-slip boundary conditions are desired on the boundary, leading to nearly identical converge rates for different parameter values as indicated by Figure~\ref{fig:rates_poisseuille}. Note, that as the mesh is refined on the edges, the problem changes numerically to a Stokes-type problem near the boundary since the mesh size $h$ drops locally below the parameter $t$. Different boundary conditions between porous media and free-flow regions are discussed from a more physical viewpoint in e.g.~\cite{EFHL08}.

\subsection{Example with realistic material parameters}
 
In this final section we consider the Society of Petroleum Engineers test case SPE10~\cite{spe10} with realistic porosity and permeability data for a heteregeneous oil reservoir. Following~\cite{Popov09}, we make the common choice $\tilde \mu = \mu$. We consider a single layer flow as a two-dimensional flow problem. For the outflow quantities to have meaningful units, we assume a thickness of $2$ ft for the layer. The dimensions of the problem are $2200 \times 1200$ ft, the fluid is chosen as e.g. water with a viscosity of $1$ cP. The flow is driven by a pressure on the left-hand side of the domain. The top and bottom boundaries have a no-flow boundary condition. To demonstrate the effect of the Brinkman term to the flow, we modify the permeability data by adding a rectangular streak of very high permeability with the dimensions $1100 \times 20$ ft in the middle of the domain. The advantage of the Brinkman model is the ability to model cracks or vugs by simply assigning a very high (or infinite) permeability to these parts of the domain. 

In the numerical tests the performance of the a posteriori estimator for adaptive refinement with extremely heterogenous data with vast differences in magnitude is studied. We start all of the computations from a coarse triangulation of the domain with 328 triangles and 1352 total degrees of freedom. This results in an initial element diameter of approximately 150 ft. In our tests we employ layer 68 from the SPE10 dataset, and assume the permeability to be as in equation~\eqref{eq:perm_diag}. 

In the first test case we consider the original dataset. In addition, we add a streak with a very high permeability of $10^6$ Darcy representing e.g. a gravel-filled crack inside the domain in two slightly different orientations, cf. Figure~\ref{fig:perms}. We also compute the net flow exiting the domain through the right boundary. We choose a simple refinement strategy in which we always refine one percent of the elements in which the estimator attains the largest values. The edgewise indicators are divided evenly to neighbouring elements. The refinements are performed until the limit of 100 000 degrees of freedom is reached. Figure~\ref{fig:flow_nostreak} demonstrates that the a posteriori estimator is able to capture the fine details of the non-modified permeability field. In Figures~\ref{fig:flow_streak} and~\ref{fig:flow_tilt} it is clearly visible, how the method is able to find the non-piercing permeability streak, even if the streak is outside the natural flow path. In particular, in Figure~\ref{fig:flow_streak} it is clearly visible how the flow paths to and from the high permeability region are found by the refinement procedure. In Figure~\ref{fig:flowrate1} we compare the flow rate for the different permeability configurations. As can be seen, the flow rate is stabilized as the mesh refinement procedure proceeds.

In the second test we modify the permeability field by adding a through-domain crack 20 feet wide into the domain by extending the streak from the first test case and we reduce the pressure loading to $1$ Pa. The permeability is set to $10^{15}$ Darcy in the crack -- thus modelling in essence a void space. We compare the Brinkman and Darcy models. In this case we use a modified refinement strategy to speed up the adaptive process. In the first three refinements we refine 15 percent of the elements, in the next three 10 percent, followed by 5 percent in the next three and 2 percent in all of the subsequent refinements. We stop refining once 100 000 degrees of freedom are attained. Clearly Figures~\ref{fig:piercing_bman} and~\ref{fig:piercing_darcy} indicate that both for the Darcy (i.e. $t=0$) and Brinkman models the indicator finds the piercing crack. However, as can be seen from the net flow results on Figure~\ref{fig:flowrate2}, the Darcy model overestimates the flow by many orders of magnitude due to the absence of the viscous forces. Also the velocity of the fluid grows unnaturally large for the Darcy model, but in the Brinkman model the velocity stays reasonable. The net flow rate results also clearly indicate how the refinement proceeds gradually until the piercing crack is found and the flow rate jumps, after which it is again stabilized.

Overall, simulating fine cracks requires extremely fine mesh around the cracks, as demonstrated by the mesh density plots. Thus uniform refinement would yield systems too large to practically solve, in particular in 3D. We also investigated the value of the ratio $\frac{t}{\sigma_K h_K}$ for all of the Brinkman simulations. When this ratio is greater than unity, we are numerically in the Stokes regime. Clearly, in the case of high permeability cracks we obtain flow domains in which we actually move numerically into the Stokes regime in the Brinkman model when modelling fractures and cracks.

\section{Conclusions}

We were able to numerically demonstrate the theoretical results for the Darcy-based method of~\cite{KSbman} for solving the Brinkman equation. Furthermore a hybridization technique was introduced for the whole system, which might prove useful for handling very large systems with the domain decomposition method. The hybridized method was also shown, both theoretically and numerically, to possess the same convergence properties as the original problem for all values of the parameter $t$. We also demonstrated the performance of the a posteriori estimator by applying it successfully to adaptive mesh refinement, and compared the Brinkman model to the Darcy model in the framework of an oil reservoir simulation.

\clearpage

%
%

\begin{figure}[!ht]
\begin{minipage}[b]{0.45\linewidth}
\centering
\includegraphics[scale=1.0]{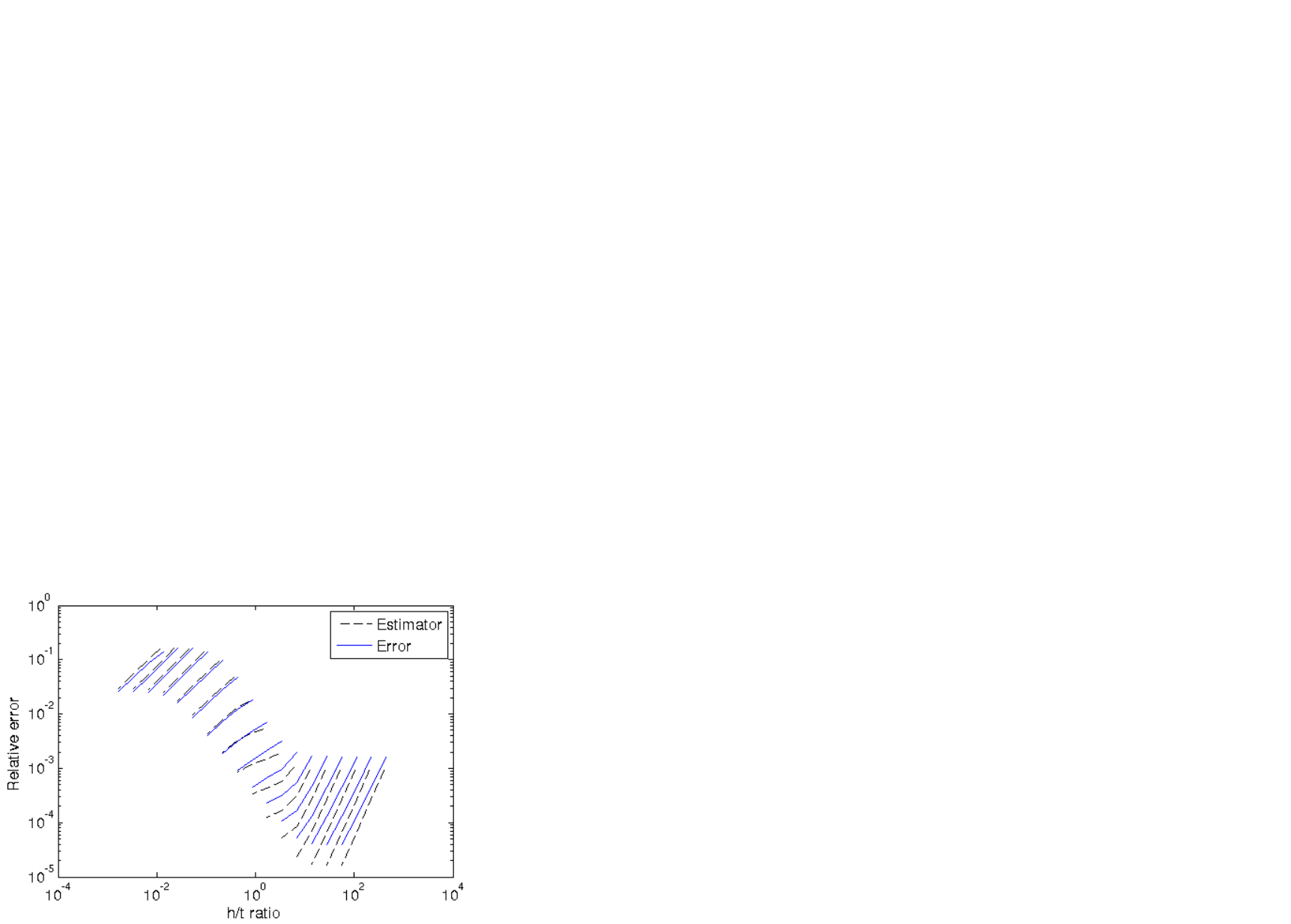}
\caption{Relative error in the mesh dependent norm. Uniform refinement for a smooth solution $\beta = 3.1$.}
\label{fig:thconv_b31}
\end{minipage}
\hspace{0.5cm}
\begin{minipage}[b]{0.45\linewidth}
\centering
\includegraphics[scale=1.0]{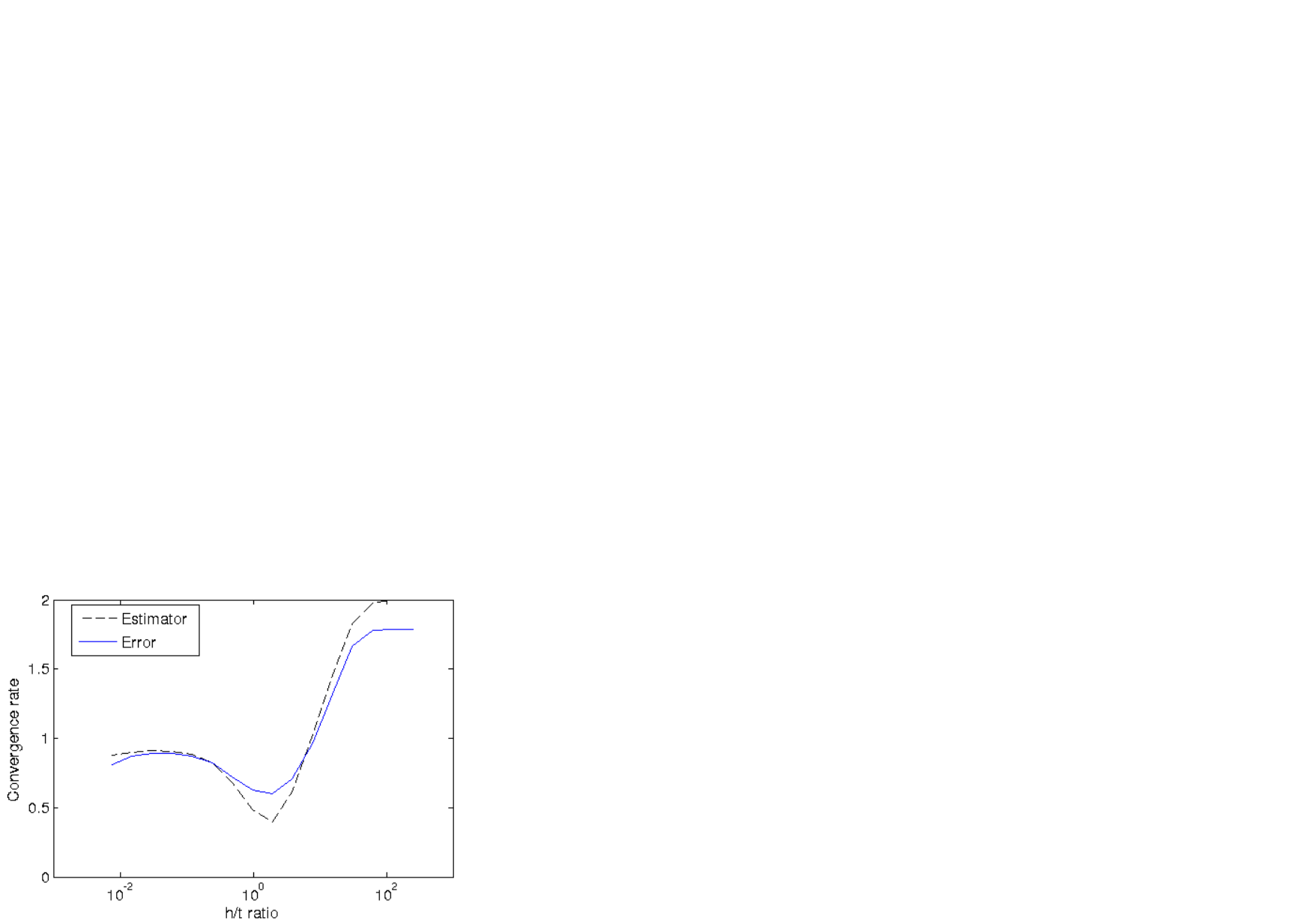}
\caption{Converge rate for different values of $t$. Uniform refinement for a smooth solution $\beta = 3.1$.}
\label{fig:rates_b31}
\end{minipage}
\end{figure}

\begin{figure}[!ht]
\begin{minipage}[b]{0.45\linewidth}
\centering
\includegraphics[scale=1.0]{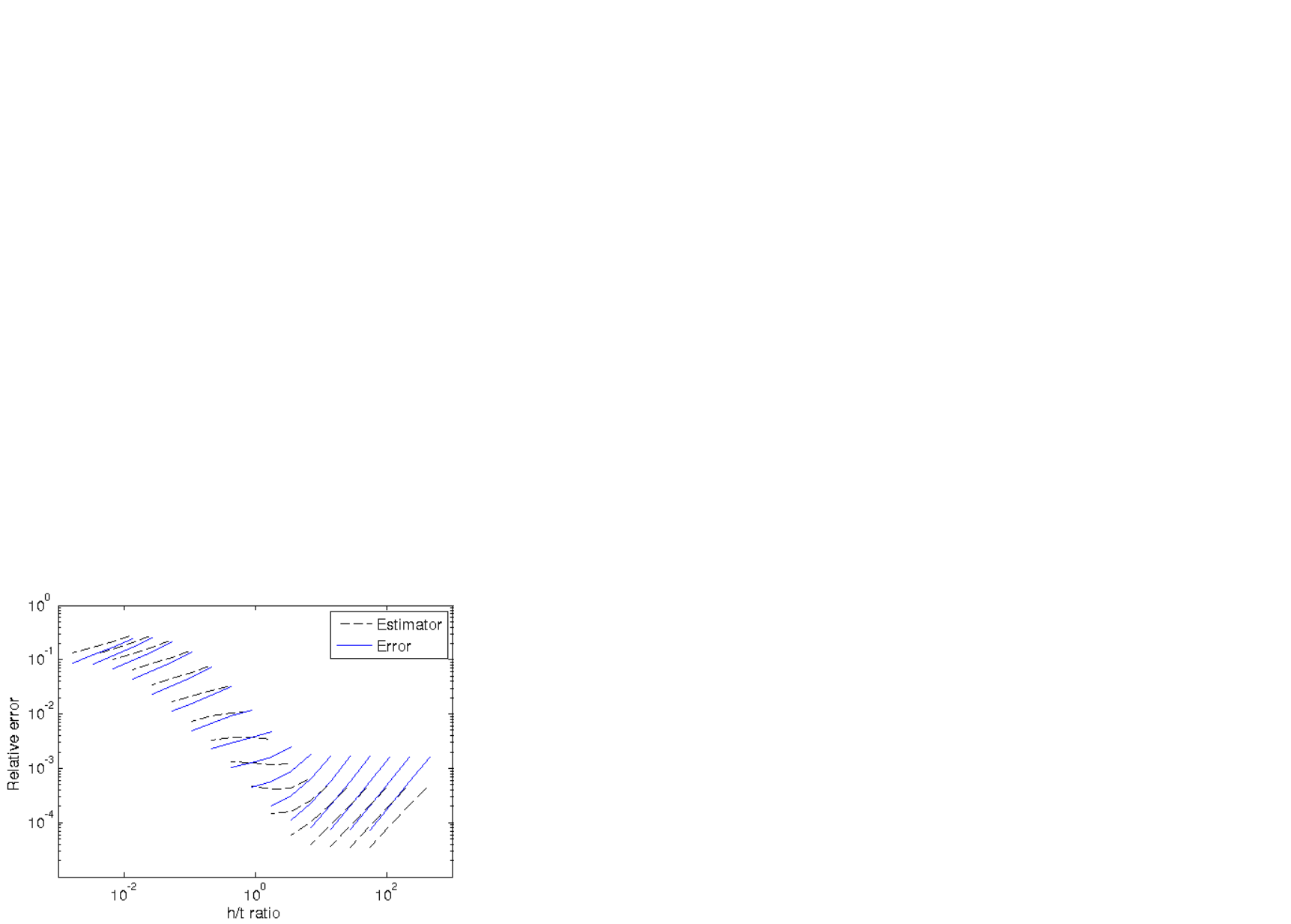}
\caption{Relative error in the mesh dependent norm. Uniform refinement for an irregular solution $\beta = 1.52$.}
\label{fig:thconv_b152}
\end{minipage}
\hspace{0.5cm}
\begin{minipage}[b]{0.45\linewidth}
\centering
\includegraphics[scale=1.0]{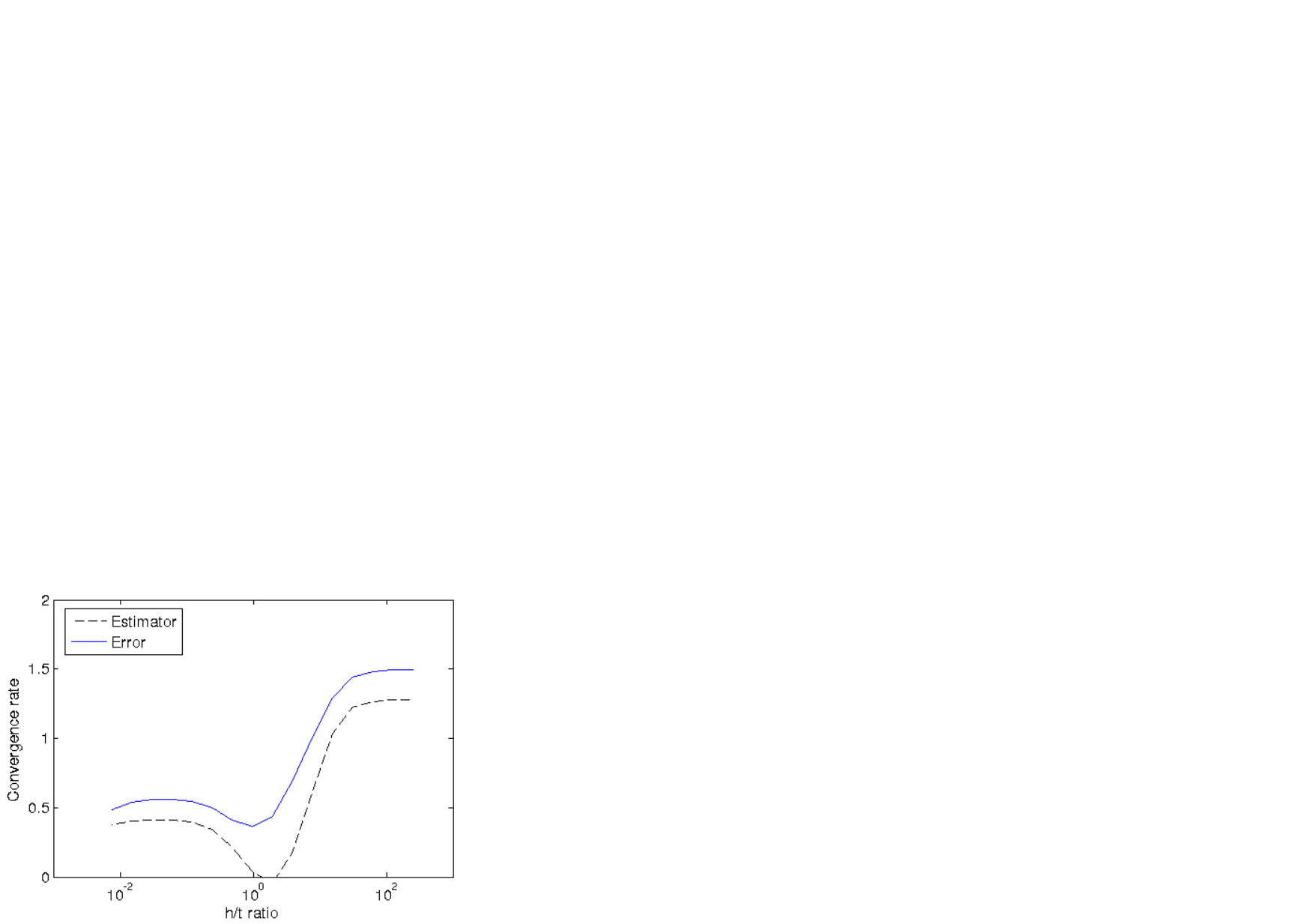}
\caption{Converge rate for different values of $t$. Uniform refinement for an irregular solution $\beta = 1.52$.}
\label{fig:rates_b152}
\end{minipage}
\end{figure}

\begin{figure}[!ht]
\begin{minipage}[b]{0.45\linewidth}
\centering
\includegraphics[scale=1.0]{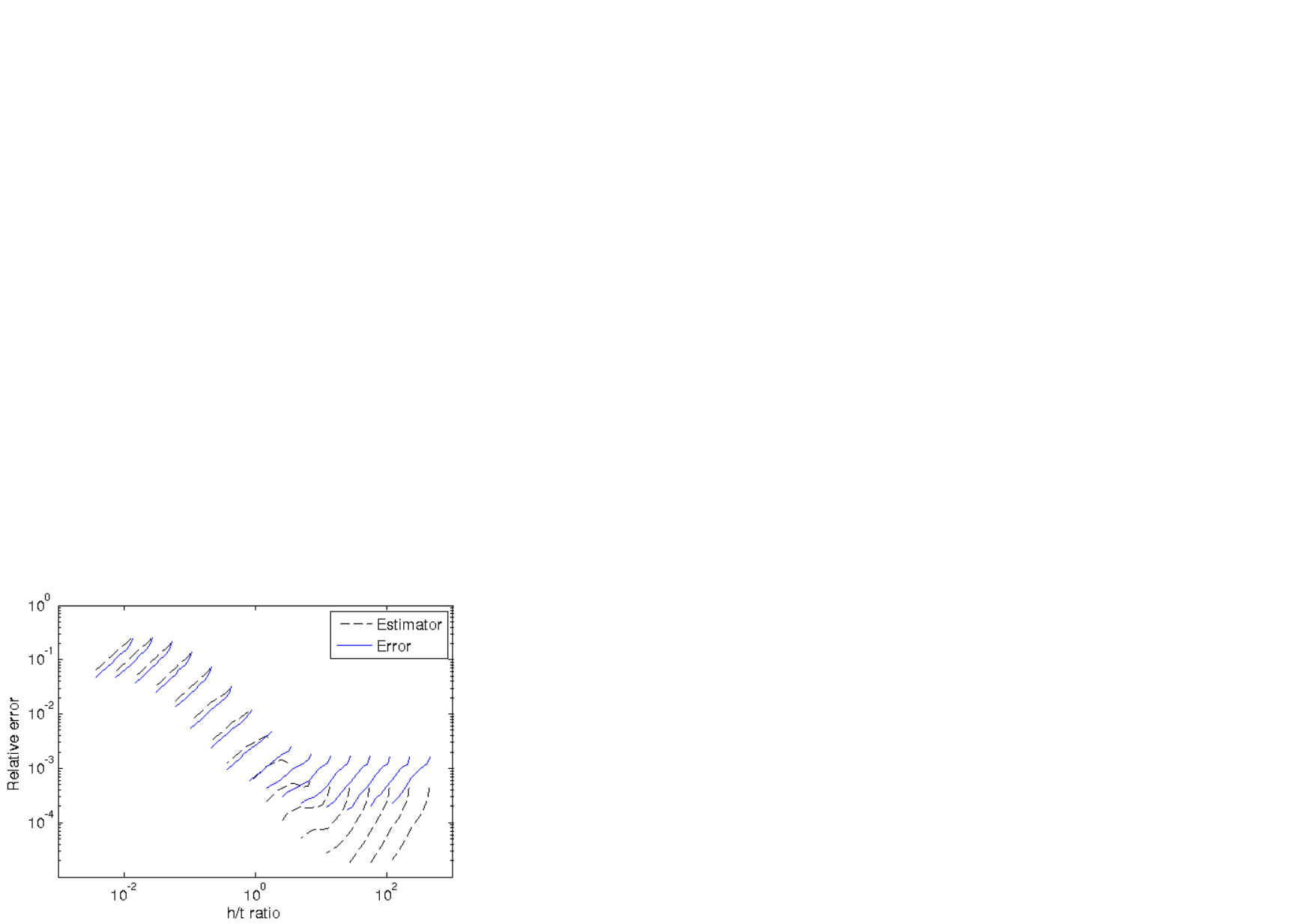}
\caption{Relative error in the mesh dependent. Adaptive refinement for an irregular solution $\beta = 1.52$.}
\label{fig:thconv_adaptiveb152}
\end{minipage}
\hspace{0.5cm}
\begin{minipage}[b]{0.45\linewidth}
\centering
\includegraphics[scale=1.0]{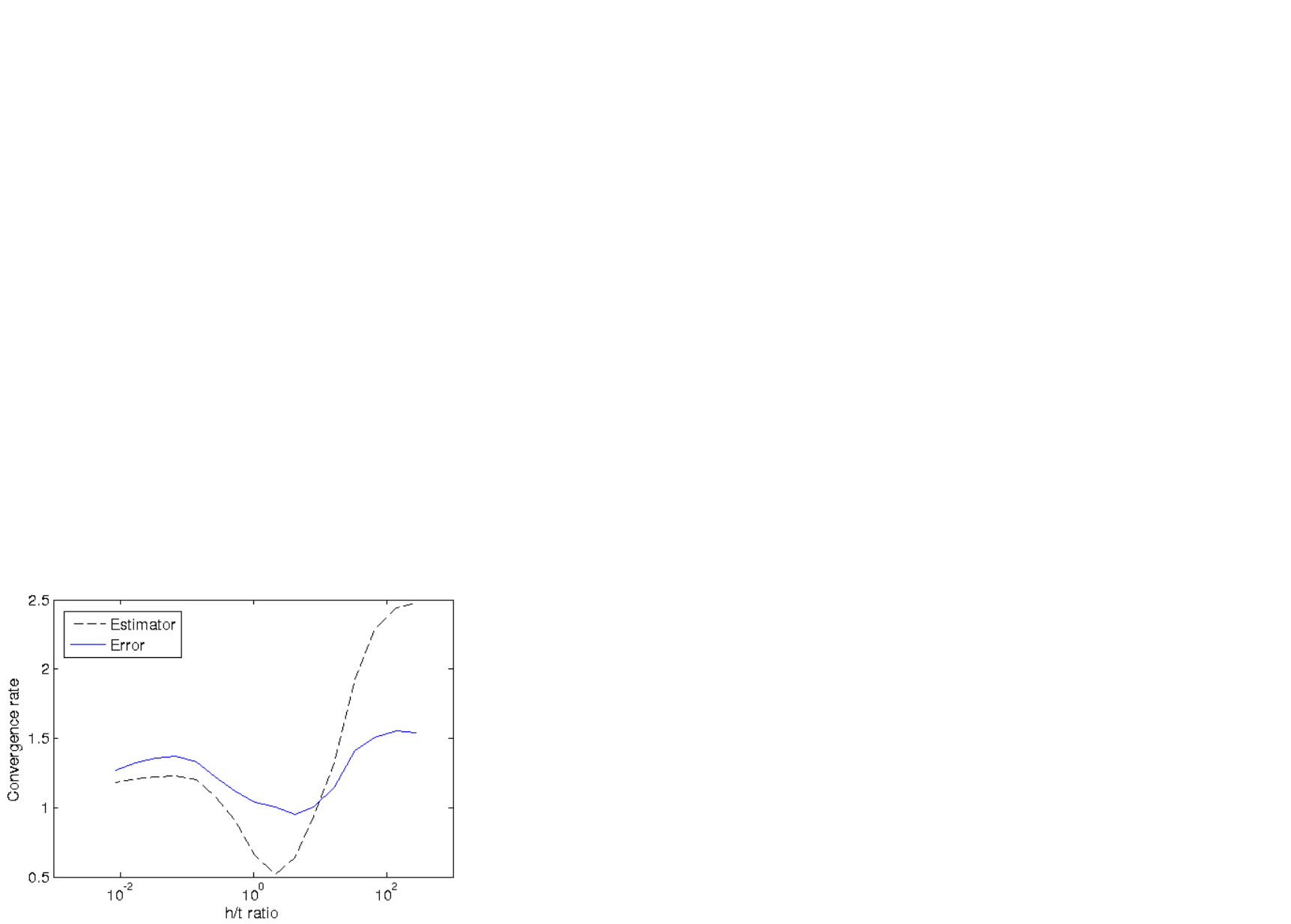}
\caption{Converge rate for different values of $t$. Adaptive refinement for an irregular solution $\beta = 1.52$.}
\label{fig:rates_adaptiveb152}
\end{minipage}
\end{figure}

%
%

\begin{figure}[!ht]
\begin{minipage}[b]{0.45\linewidth}
\centering
\includegraphics[scale=1.0]{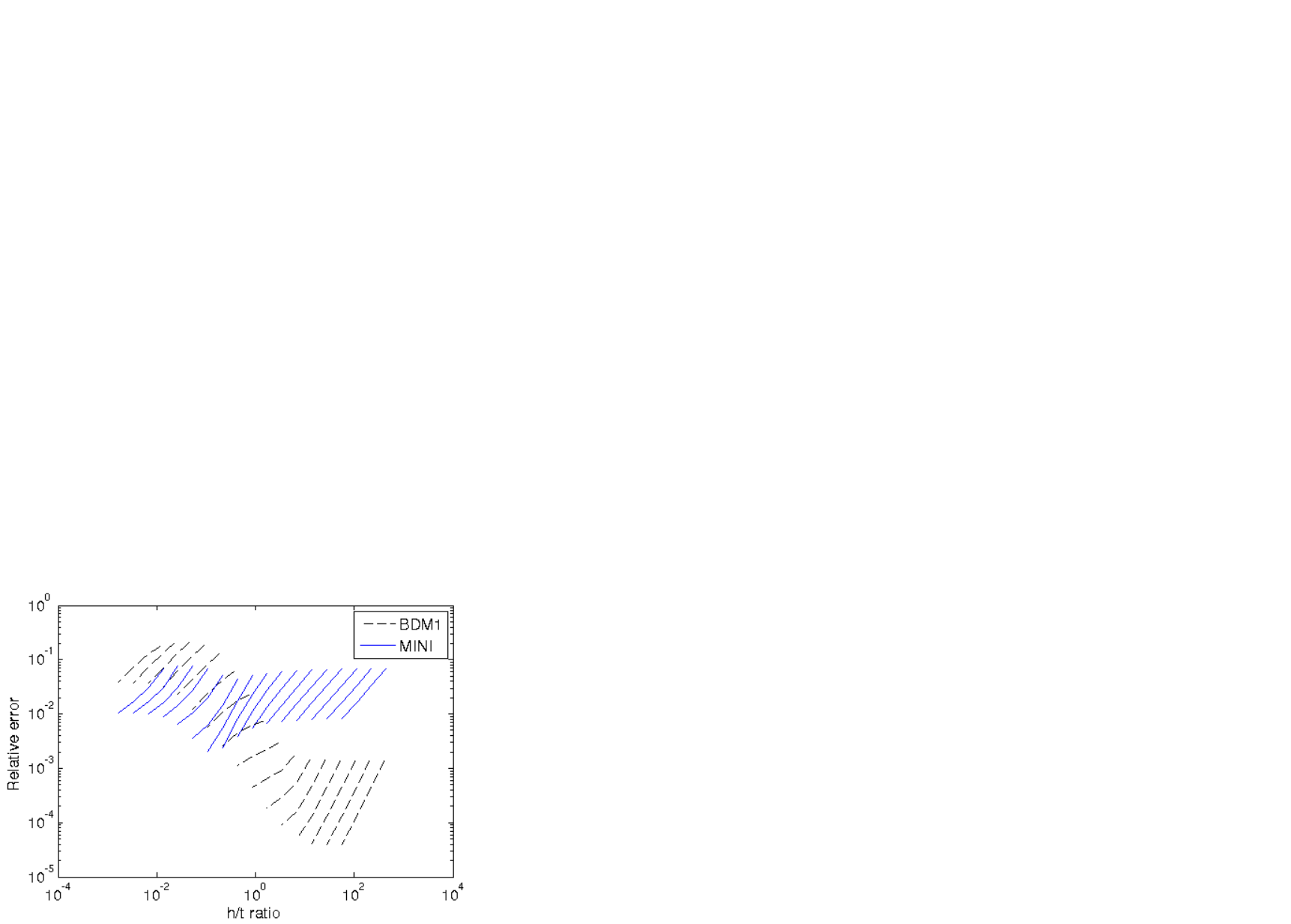}
\caption{Relative error in the mesh dependent norm. Uniform refinement for a smooth solution $\beta = 3.1$.}
\label{fig:thconv_comparisonb31}
\end{minipage}
\hspace{0.5cm}
\begin{minipage}[b]{0.45\linewidth}
\centering
\includegraphics[scale=1.0]{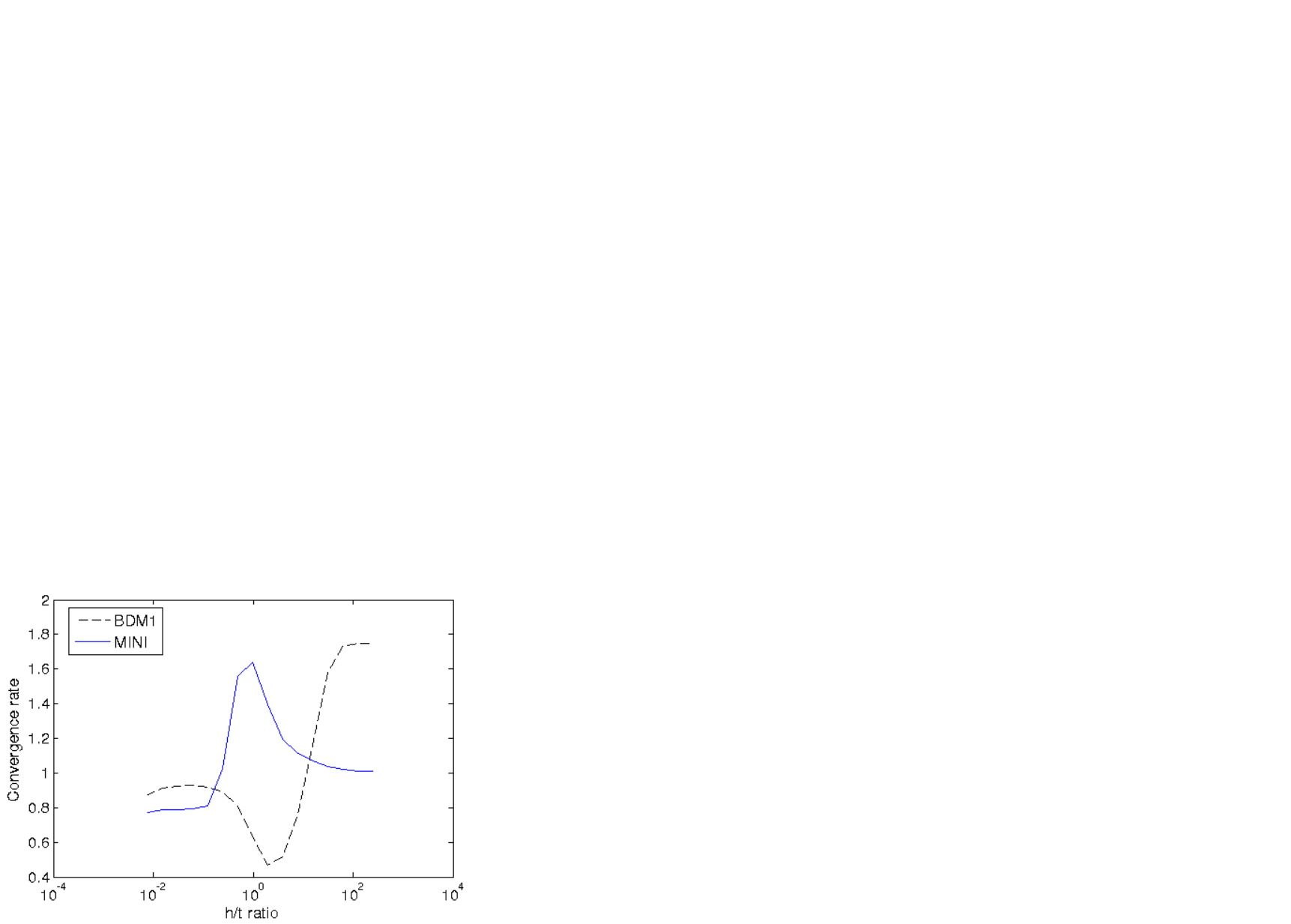}
\caption{Converge rate for different values of $t$. Uniform refinement for a smooth solution $\beta = 3.1$.}
\label{fig:rates_comparisonb31}
\end{minipage}
\end{figure}

%
%

\begin{figure}[!ht]
\begin{minipage}[b]{0.45\linewidth}
\centering
\includegraphics[scale=1.0]{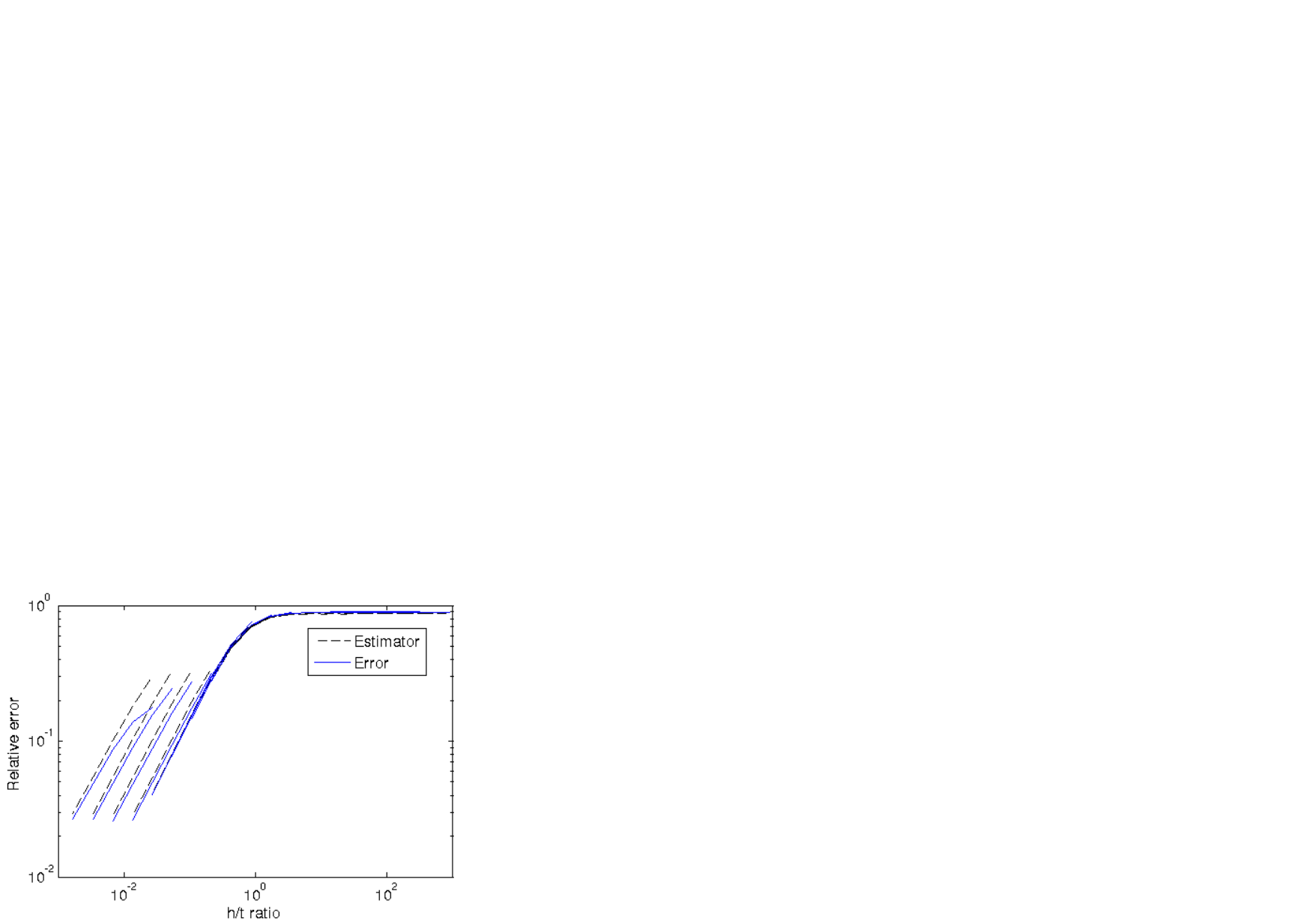}
\caption{Relative error in the mesh dependent norm without postprocessing. Uniform refinement for $\beta = 3.1$.}
\label{fig:thconv_noppb31}
\end{minipage}
\hspace{0.5cm}
\begin{minipage}[b]{0.45\linewidth}
\centering
\includegraphics[scale=1.0]{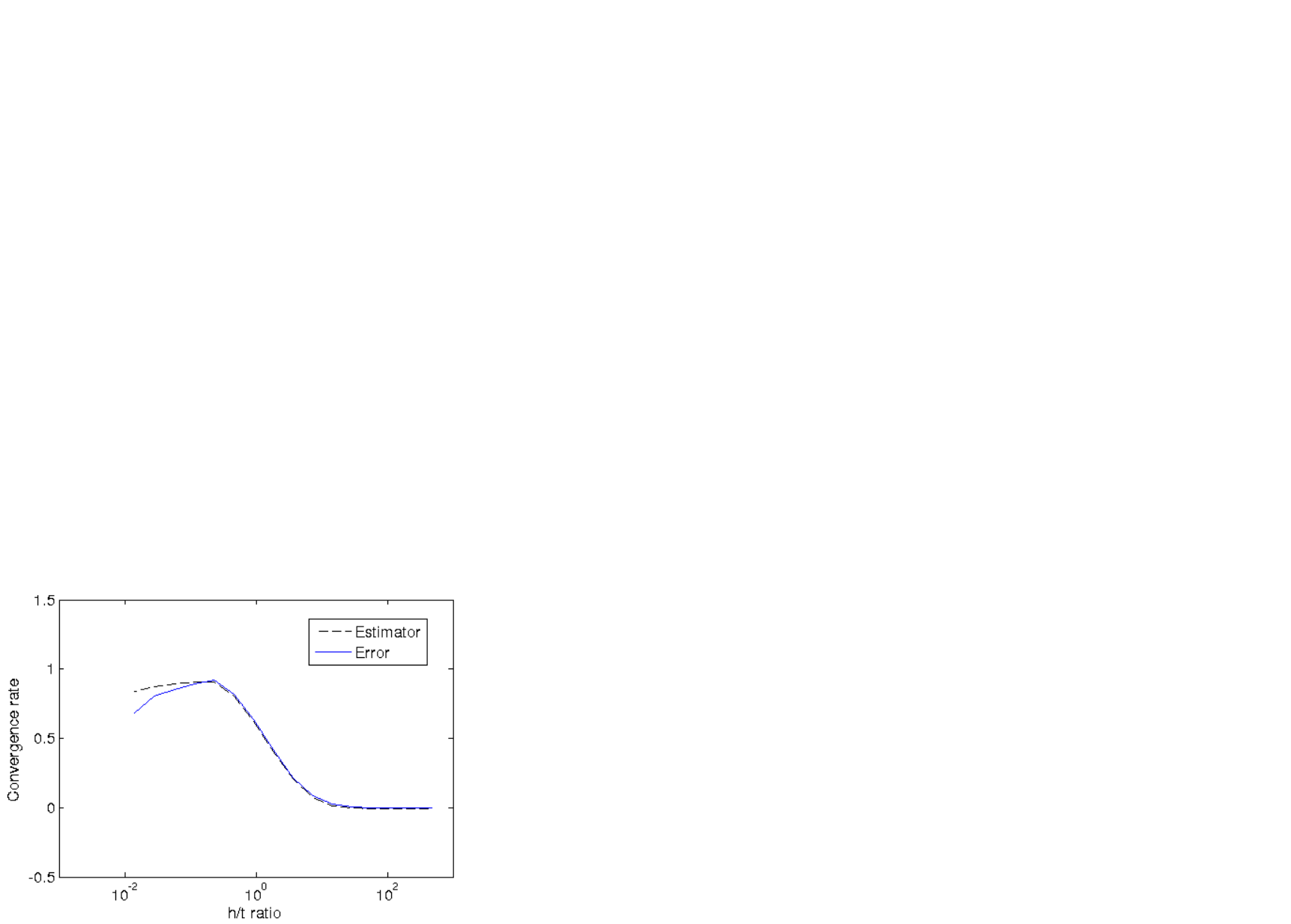}
\caption{Converge rate for different values of $t$ without postprocessing. Uniform refinement for $\beta = 3.1$.}
\label{fig:rates_noppb31}
\end{minipage}
\end{figure}

\begin{figure}[!ht]
\begin{minipage}[b]{0.45\linewidth}
\centering
\includegraphics[scale=1.0]{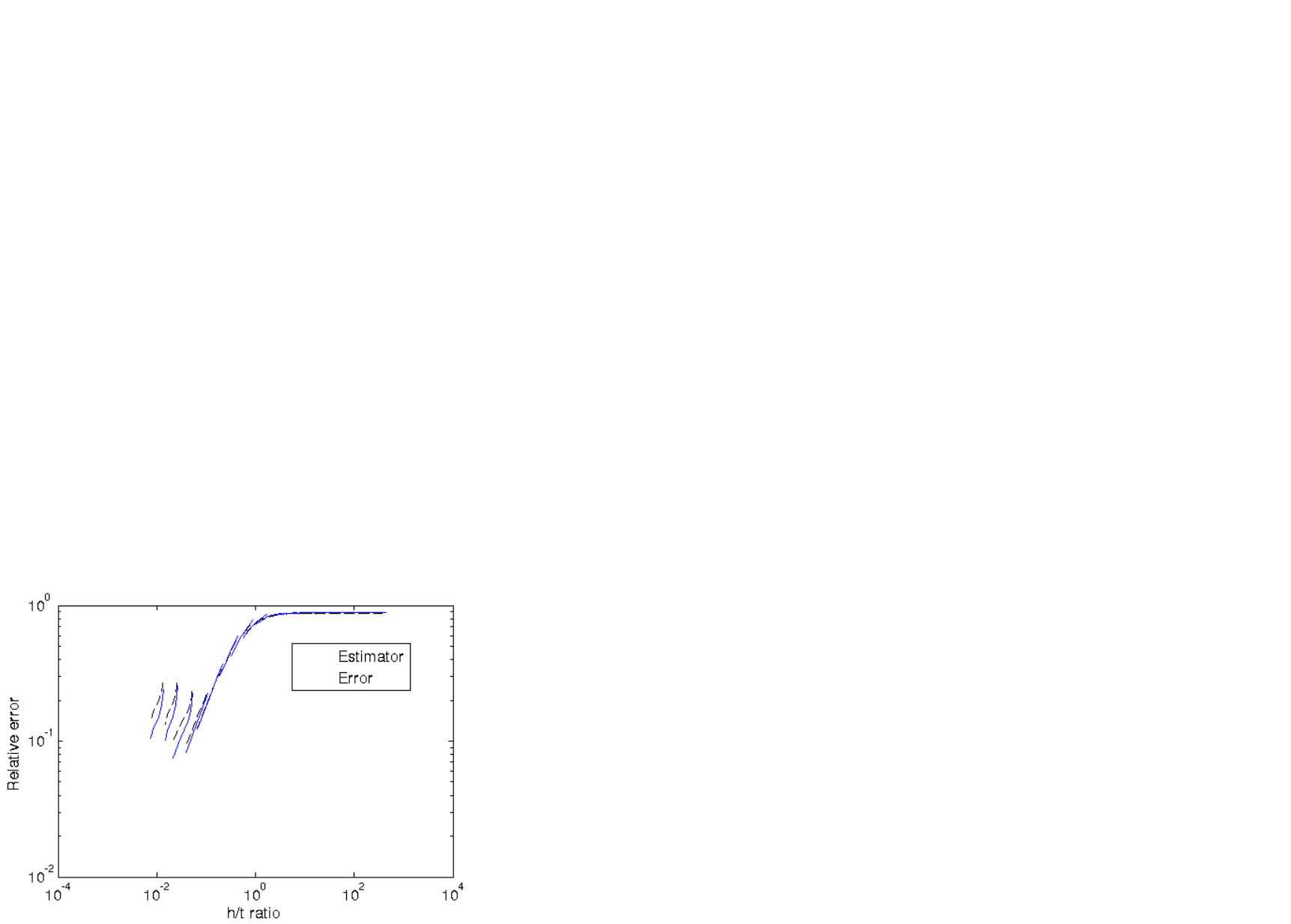}
\caption{Relative error in the mesh dependent norm without postprocessing. Adaptive refinement for $\beta = 1.52$.}
\label{fig:thconv_noppadab152}
\end{minipage}
\hspace{0.5cm}
\begin{minipage}[b]{0.45\linewidth}
\centering
\includegraphics[scale=1.0]{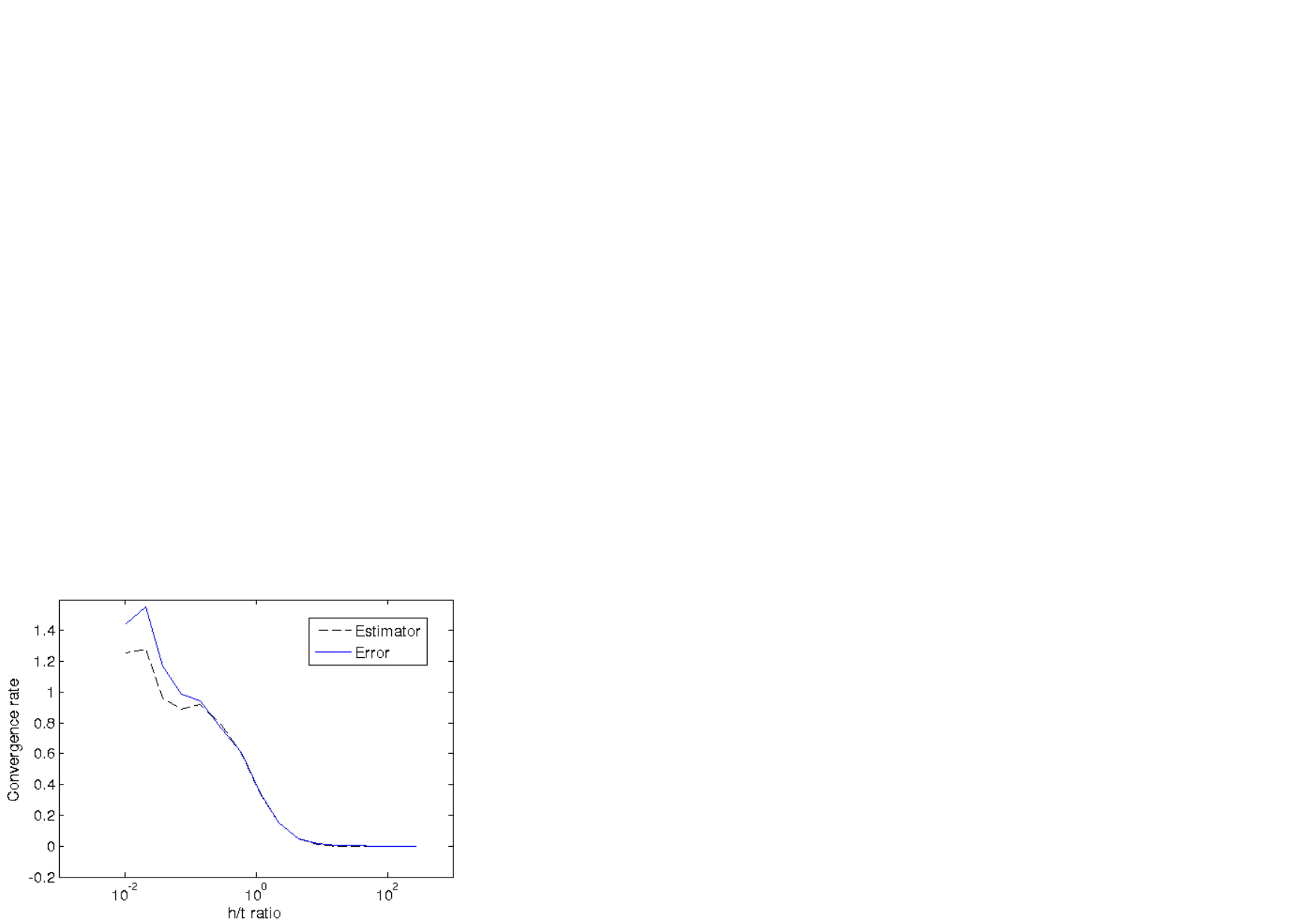}
\caption{Converge rate for different values of $t$ without postprocessing. Adaptive refinement for $\beta = 1.52$.}
\label{fig:rates_noppadab152}
\end{minipage}
\end{figure}

%
%

\begin{figure}[!ht]
\begin{minipage}[b]{0.45\linewidth}
\centering
\includegraphics[scale=1.0]{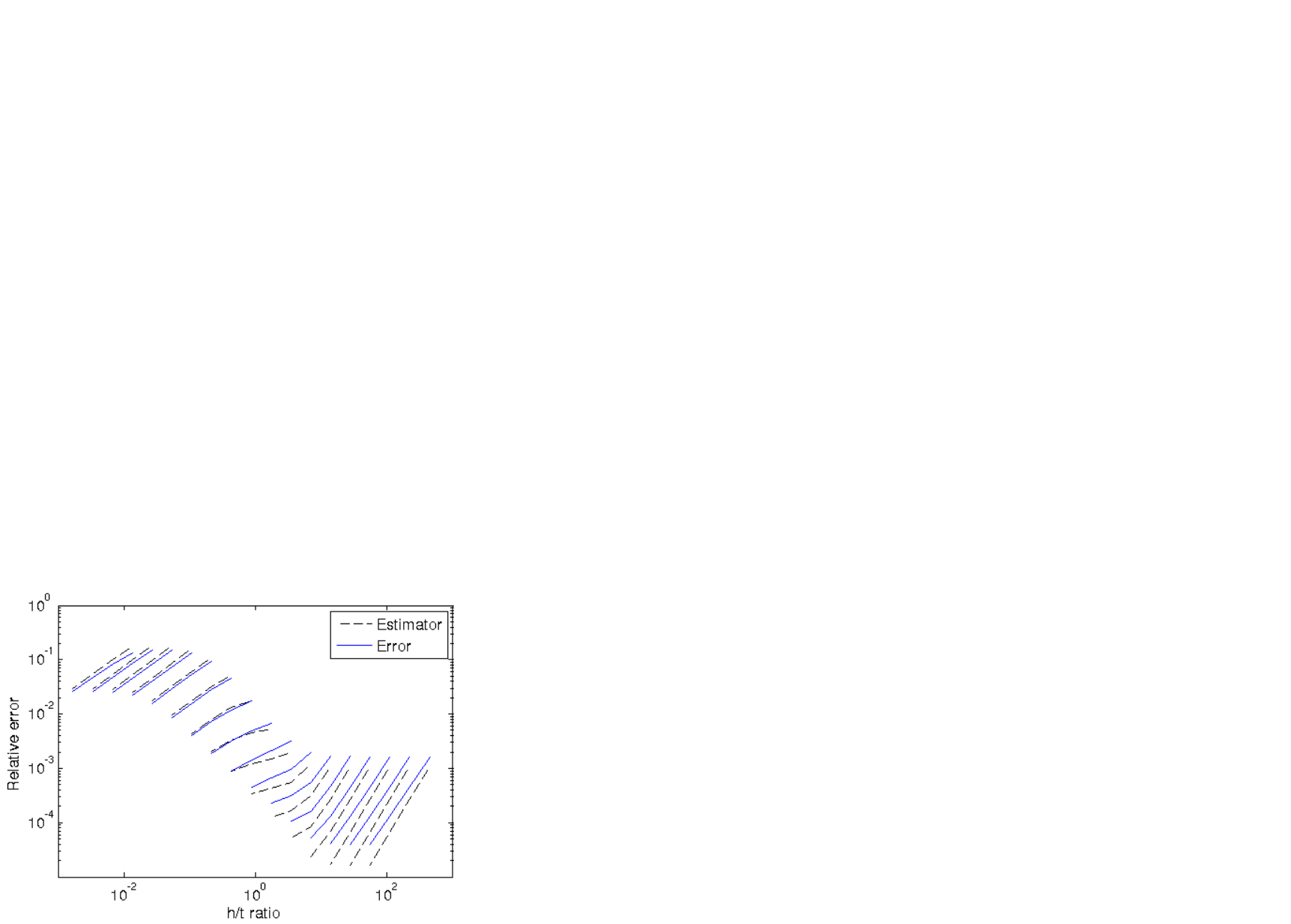}
\caption{Relative error in the mesh dependent norm for the hybridized method with uniform refinement and $\beta = 3.1$.}
\label{fig:thconv_hyb31}
\end{minipage}
\hspace{0.5cm}
\begin{minipage}[b]{0.45\linewidth}
\centering
\includegraphics[scale=1.0]{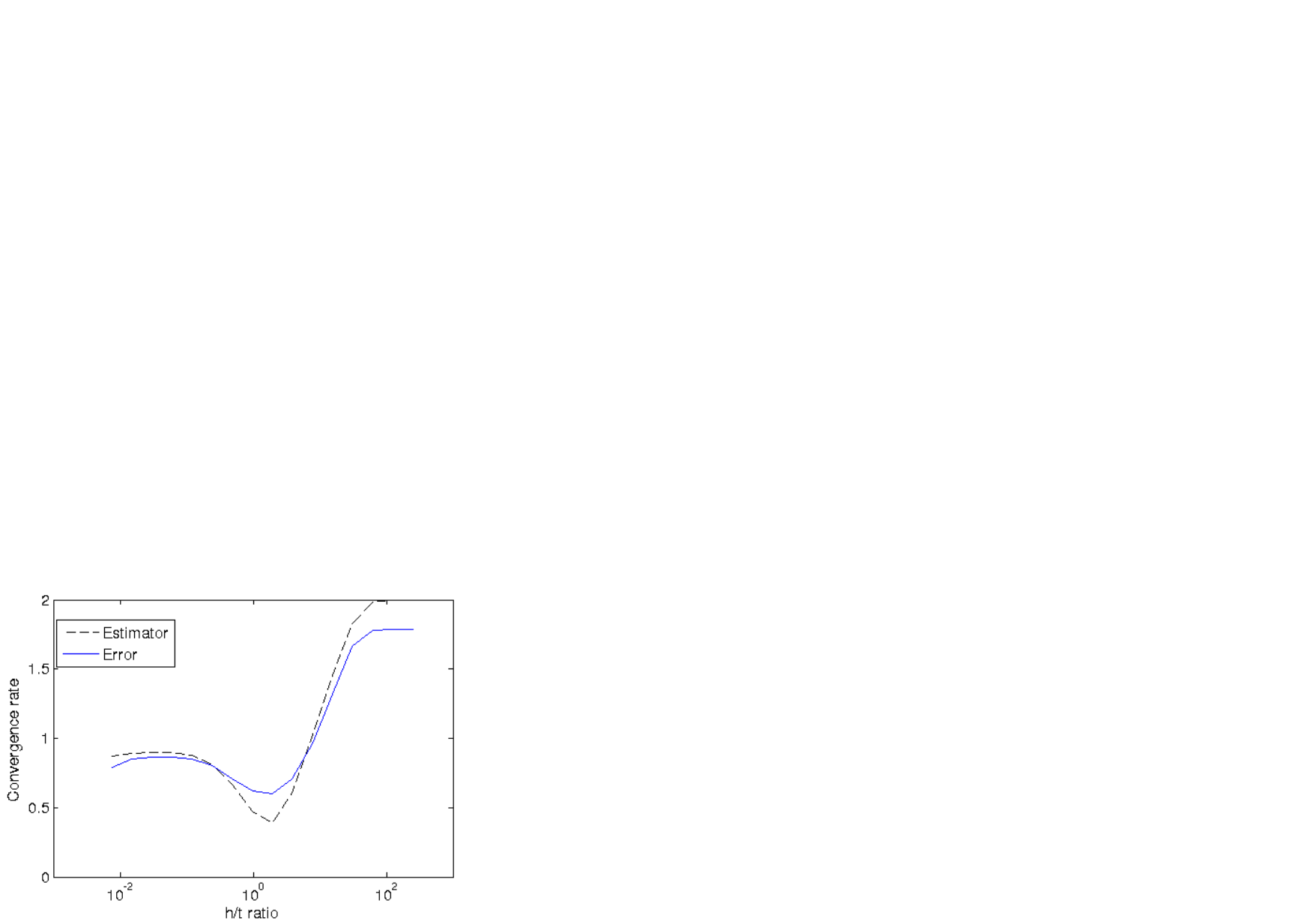}
\caption{Converge rate for different values of $t$ for the hybridized method with $\beta = 3.1$.}
\label{fig:rates_hyb31}
\end{minipage}
\end{figure}

\begin{figure}[!ht]
\begin{minipage}[b]{0.45\linewidth}
\centering
\includegraphics[scale=1.0]{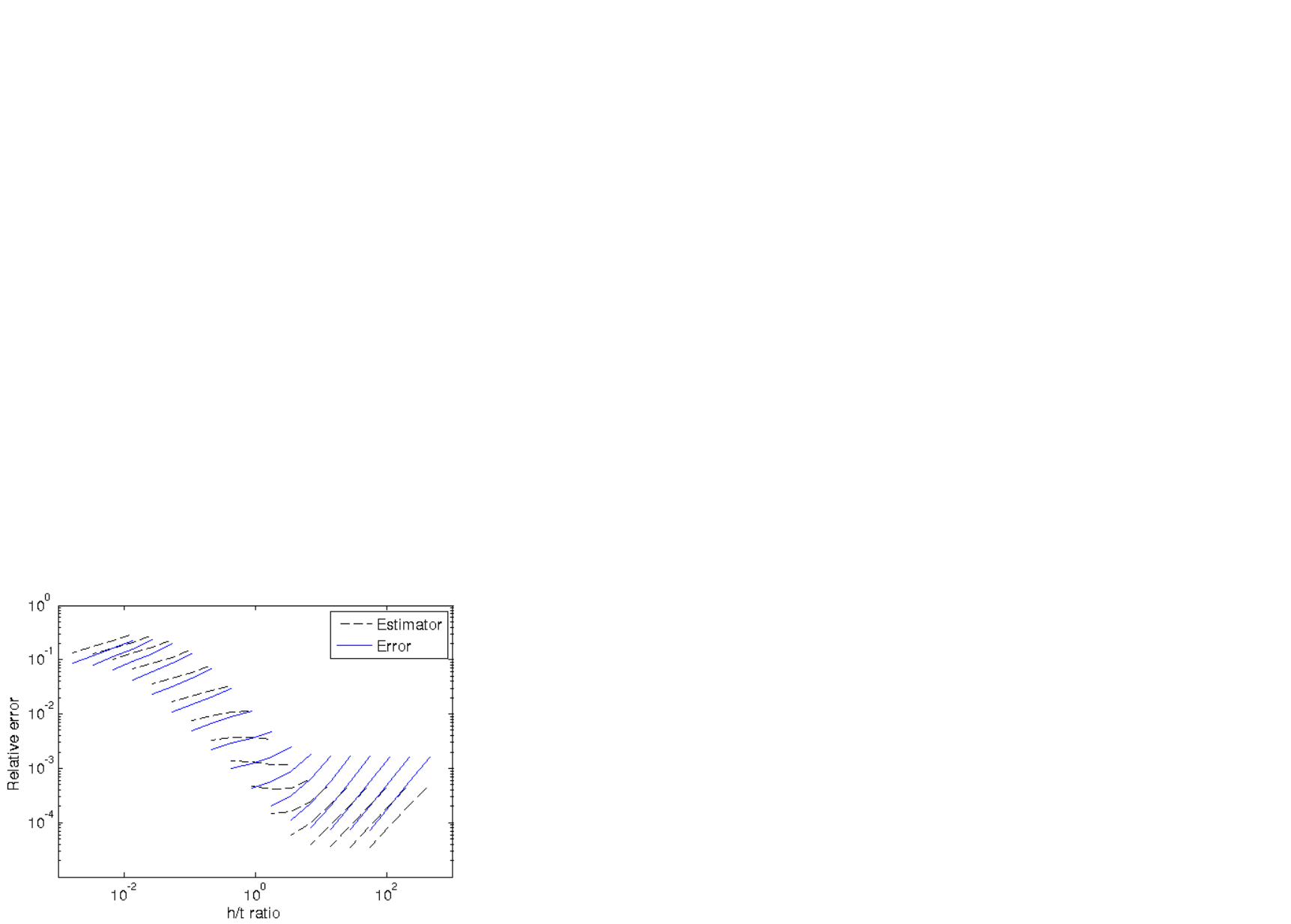}
\caption{Relative error in the mesh dependent norm for the hybridized method with uniform refinement and $\beta = 1.52$.}
\label{fig:thconv_hyb152}
\end{minipage}
\hspace{0.5cm}
\begin{minipage}[b]{0.45\linewidth}
\centering
\includegraphics[scale=1.0]{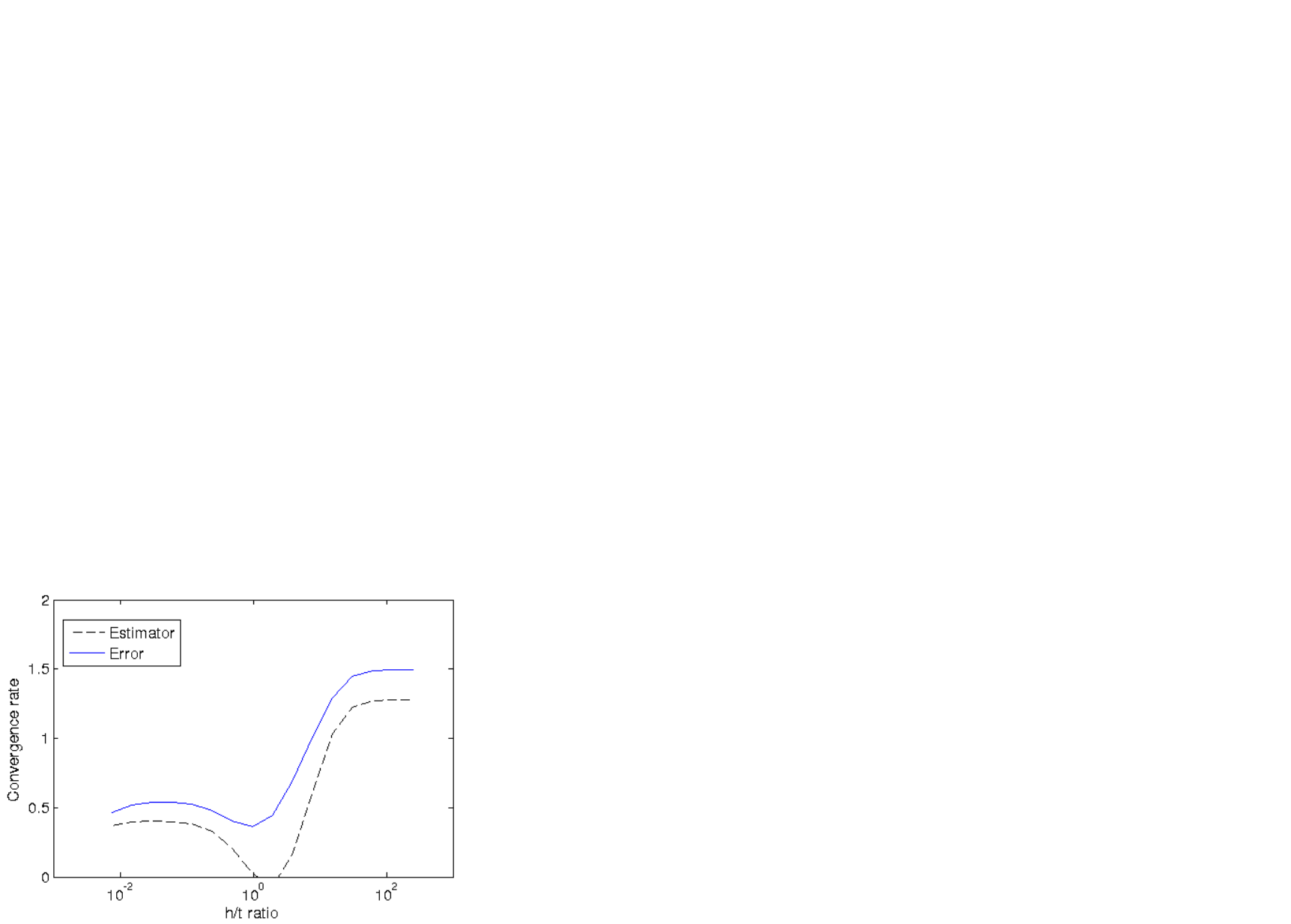}
\caption{Converge rate for different values of $t$ for the hybridized method with $\beta = 1.52$.}
\label{fig:rates_hyb152}
\end{minipage}
\end{figure}

\begin{figure}[!ht]
\begin{minipage}[b]{0.45\linewidth}
\centering
\includegraphics[scale=1.0]{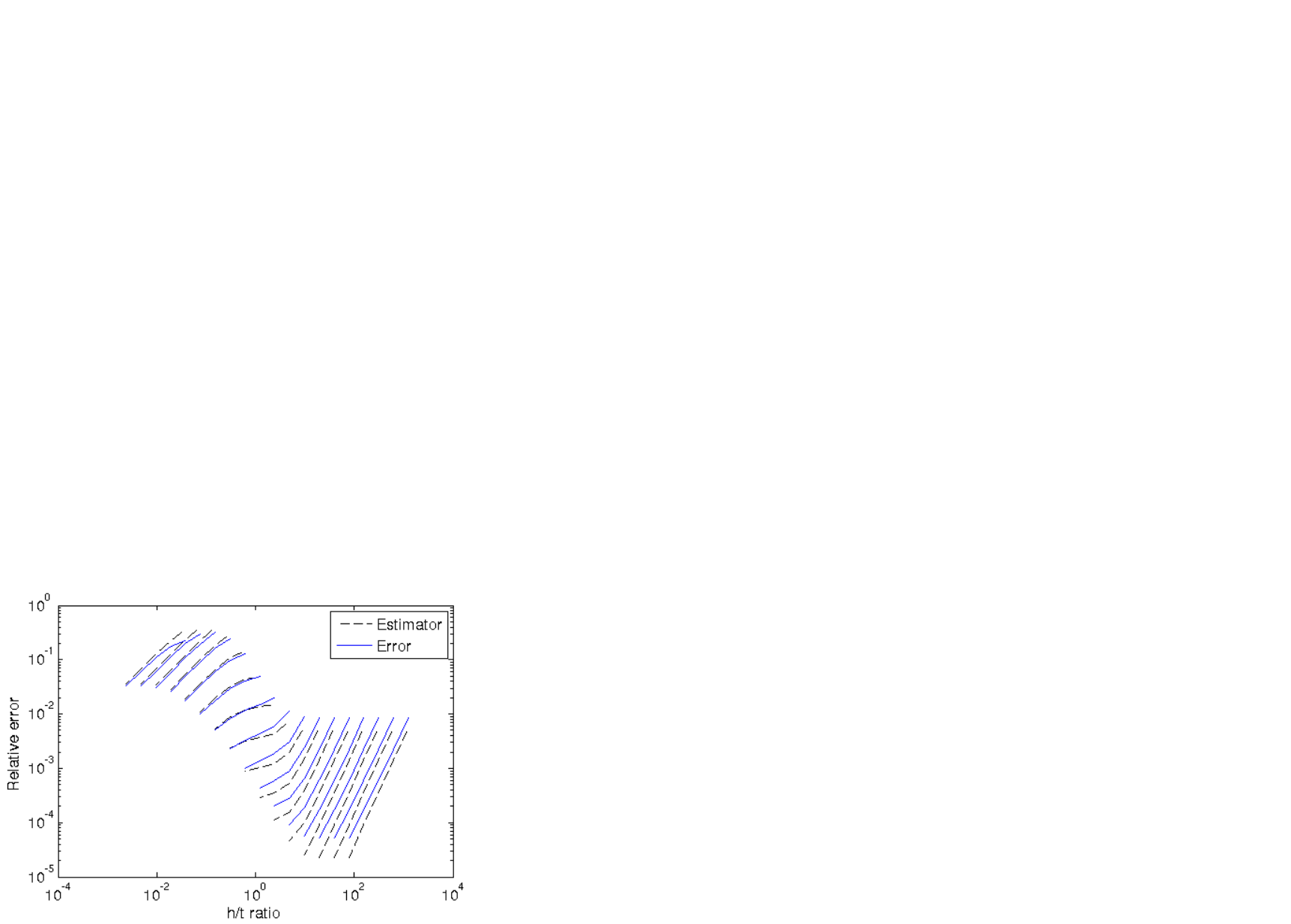}
\caption{Relative error in the mesh dependent norm for the domain decomposition with 16 subdomains and $\beta = 3.1$.}
\label{fig:thconv_dd31}
\end{minipage}
\hspace{0.5cm}
\begin{minipage}[b]{0.45\linewidth}
\centering
\includegraphics[scale=1.0]{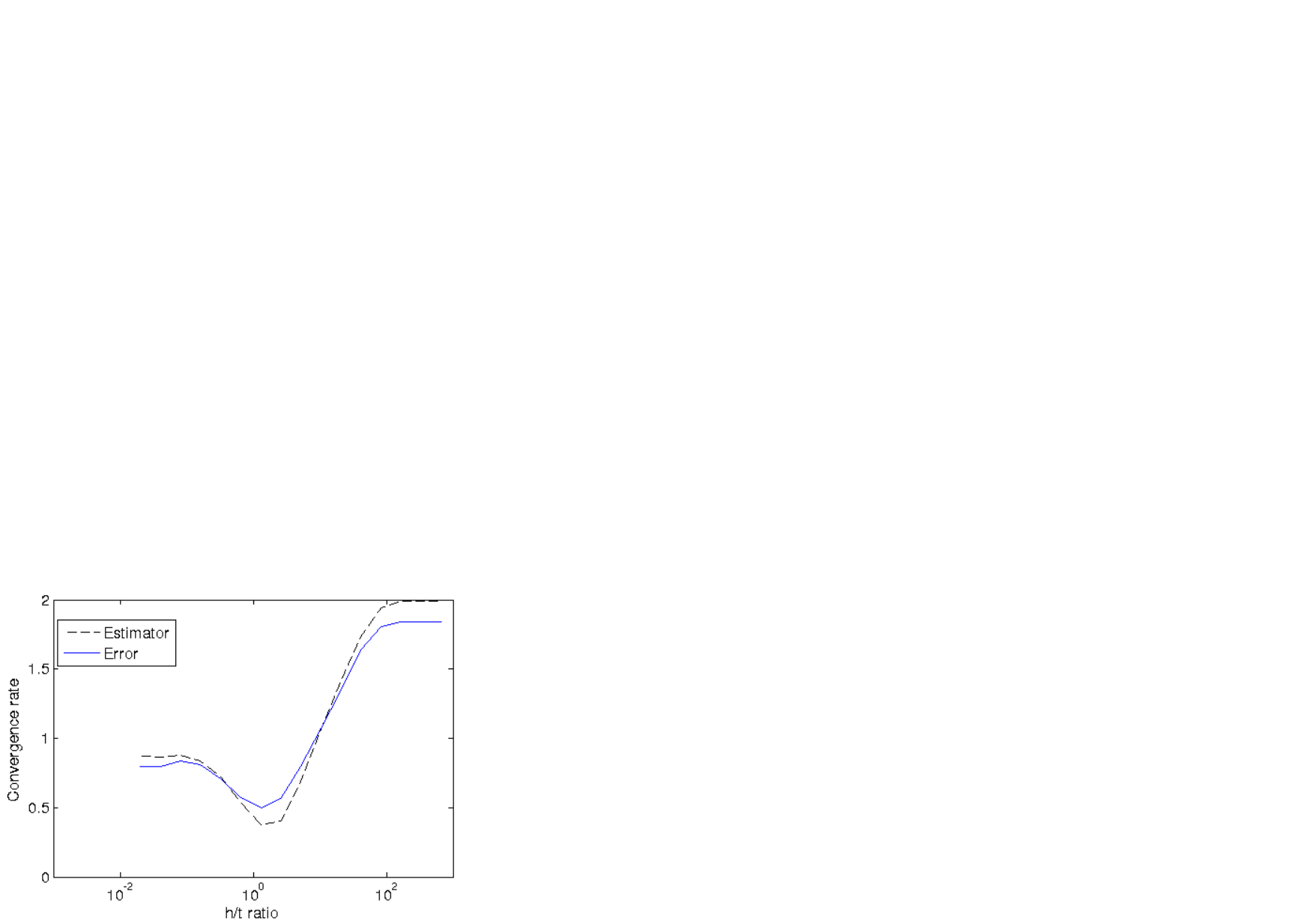}
\caption{Converge rate for different values of $t$ for the domain decomposition with 16 subdomains and $\beta = 3.1$.}
\label{fig:rates_dd31}
\end{minipage}
\end{figure}

\begin{figure}[!ht]
\begin{minipage}[b]{0.45\linewidth}
\centering
\includegraphics[scale=1.0]{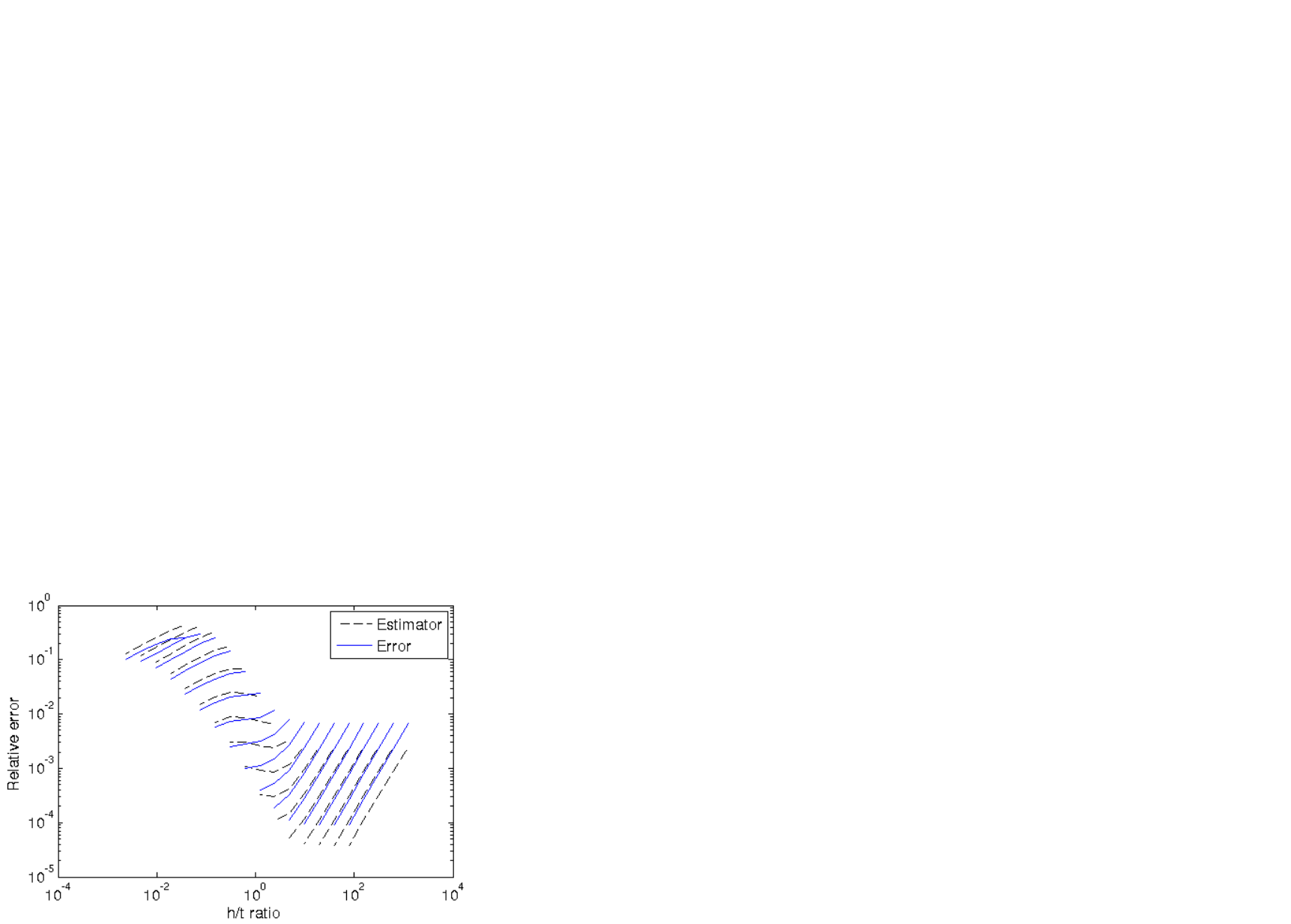}
\caption{Relative error in the mesh dependent norm for the domain decomposition with 16 subdomains and $\beta = 1.52$.}
\label{fig:thconv_dd152}
\end{minipage}
\hspace{0.5cm}
\begin{minipage}[b]{0.45\linewidth}
\centering
\includegraphics[scale=1.0]{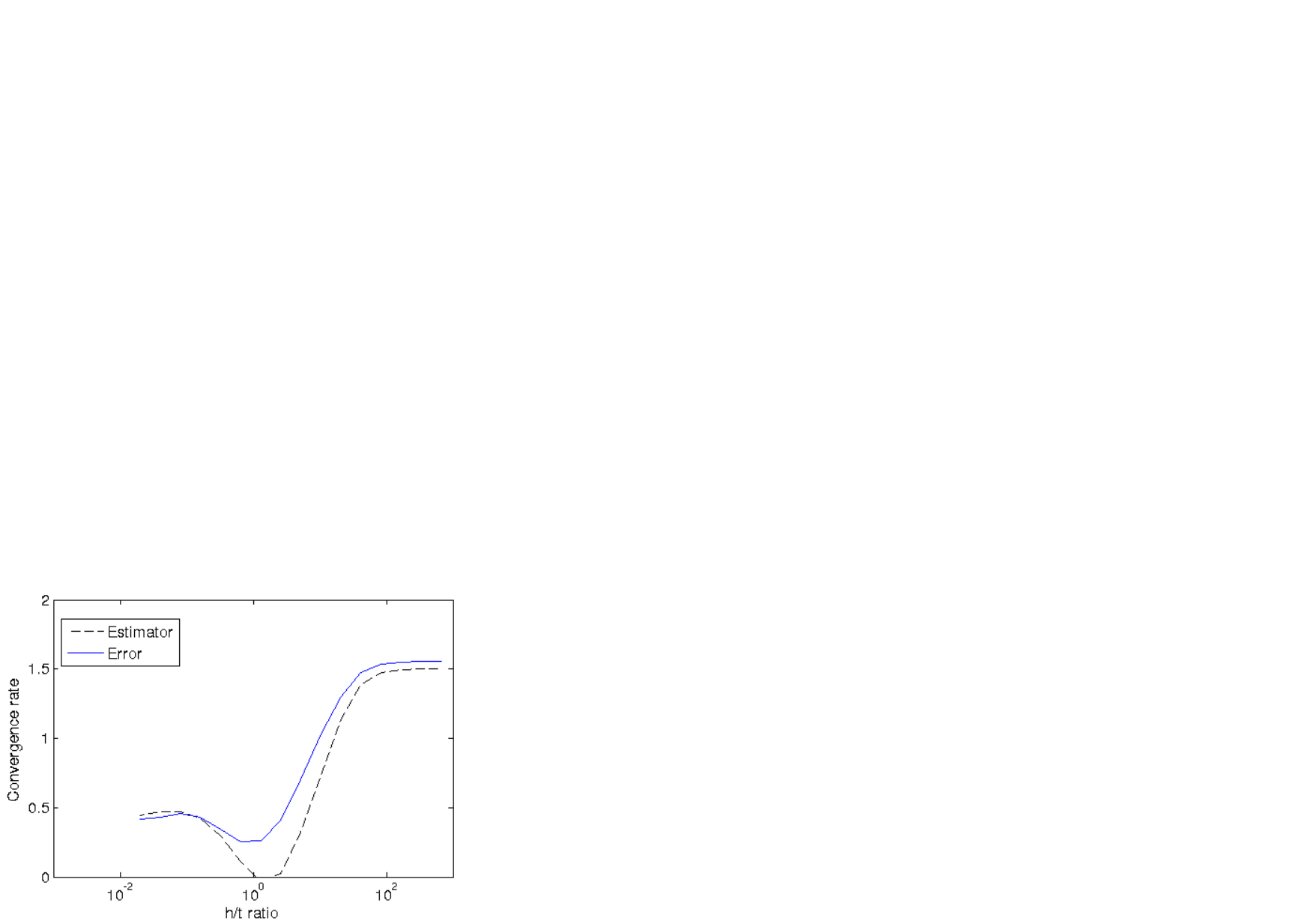}
\caption{Converge rate for different values of $t$ for the domain decomposition with 16 subdomains and $\beta = 1.52$.}
\label{fig:rates_dd152}
\end{minipage}
\end{figure}

\begin{figure}[!ht]
\centering
\hspace*{-1cm}
\includegraphics[scale=1.0]{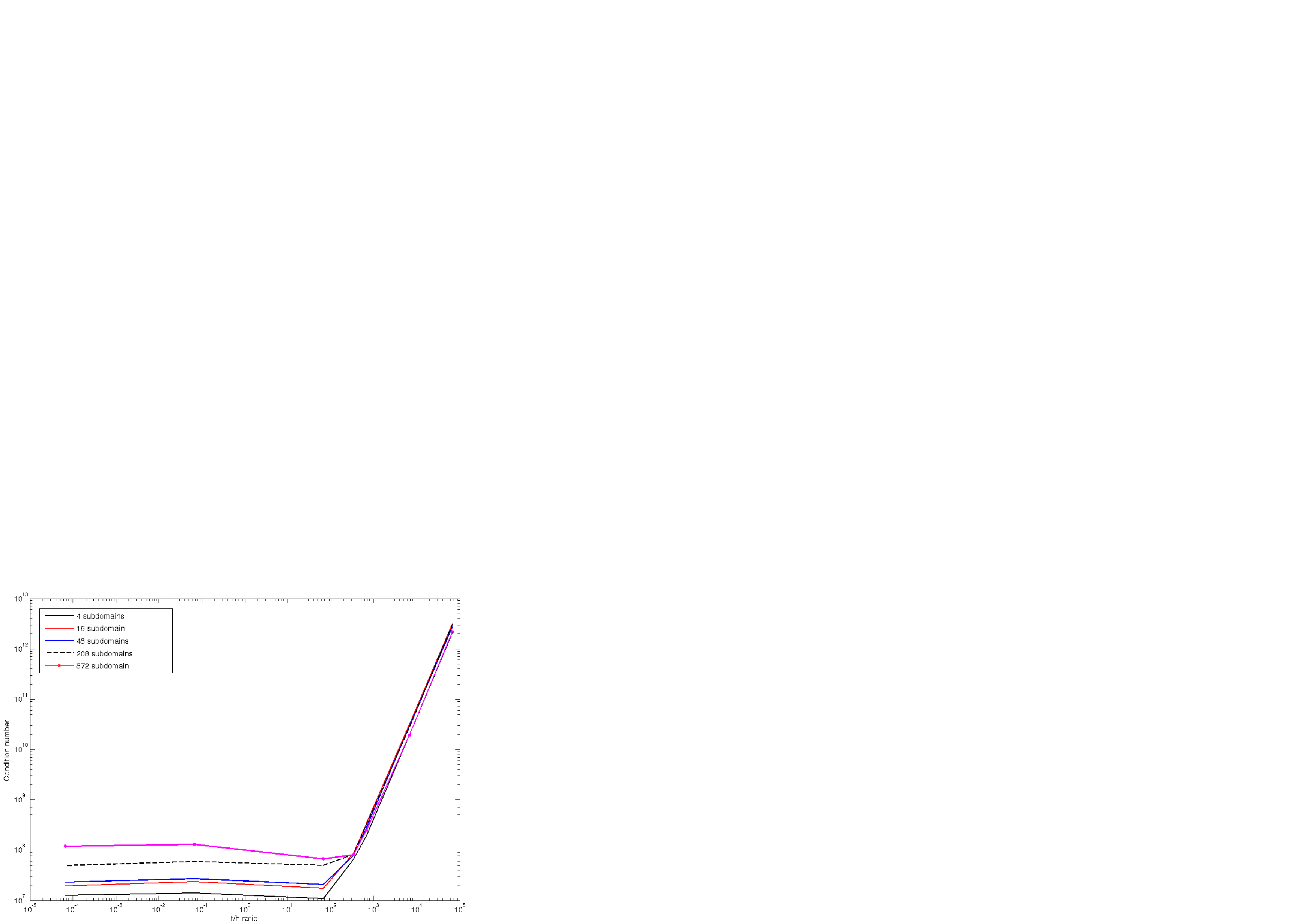}
\caption{Condition number for several subdomain divisions for different values of $t$.}
\label{fig:condtest}
\end{figure}

%
%

\begin{figure}[ht]
\begin{minipage}[b]{0.45\linewidth}
\centering
\includegraphics[scale=0.8]{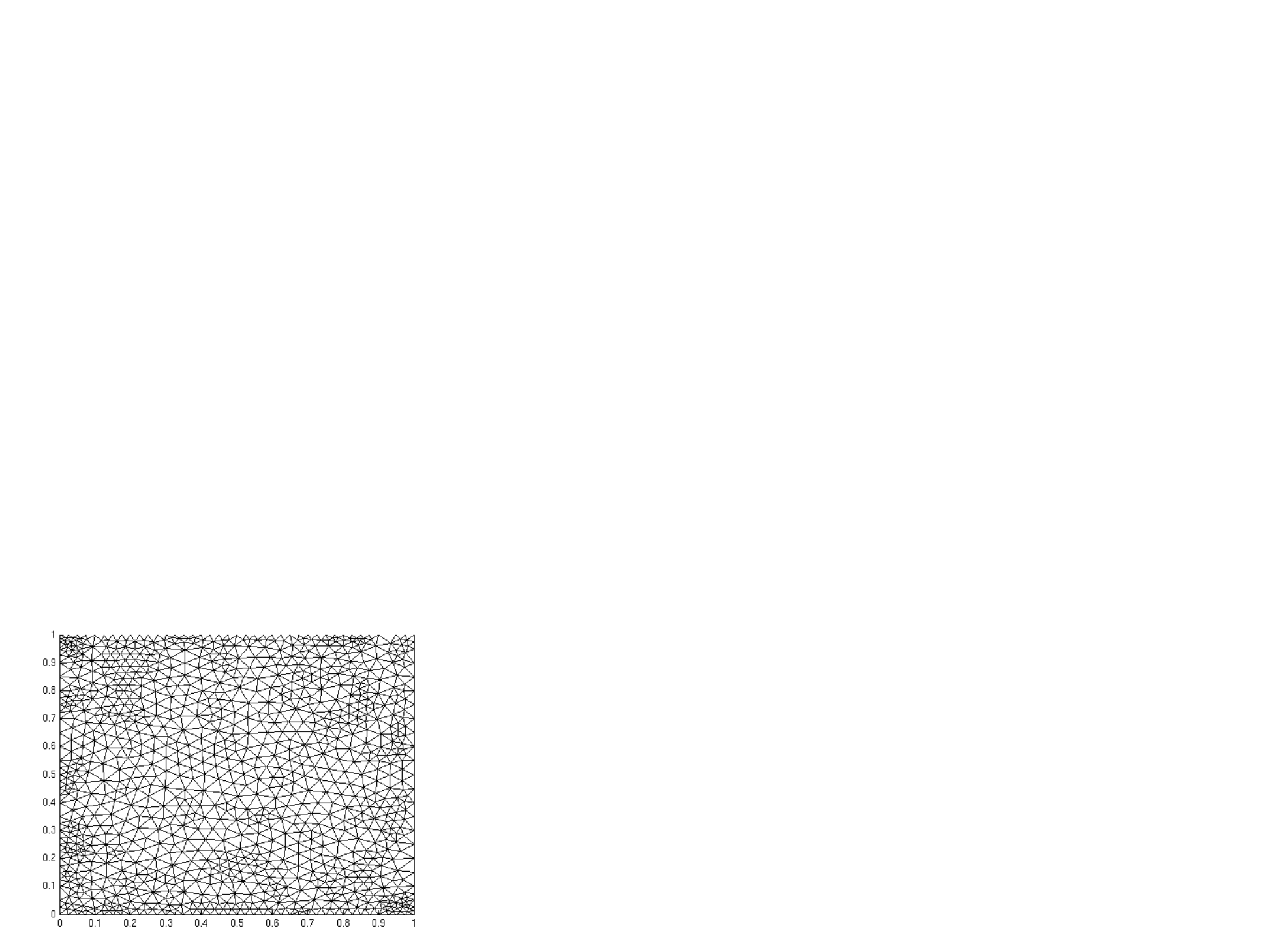}
\caption{Final mesh after adaptive refinement, $t=0.5$}
\label{fig:poismesh05}
\centering
\includegraphics[scale=0.8]{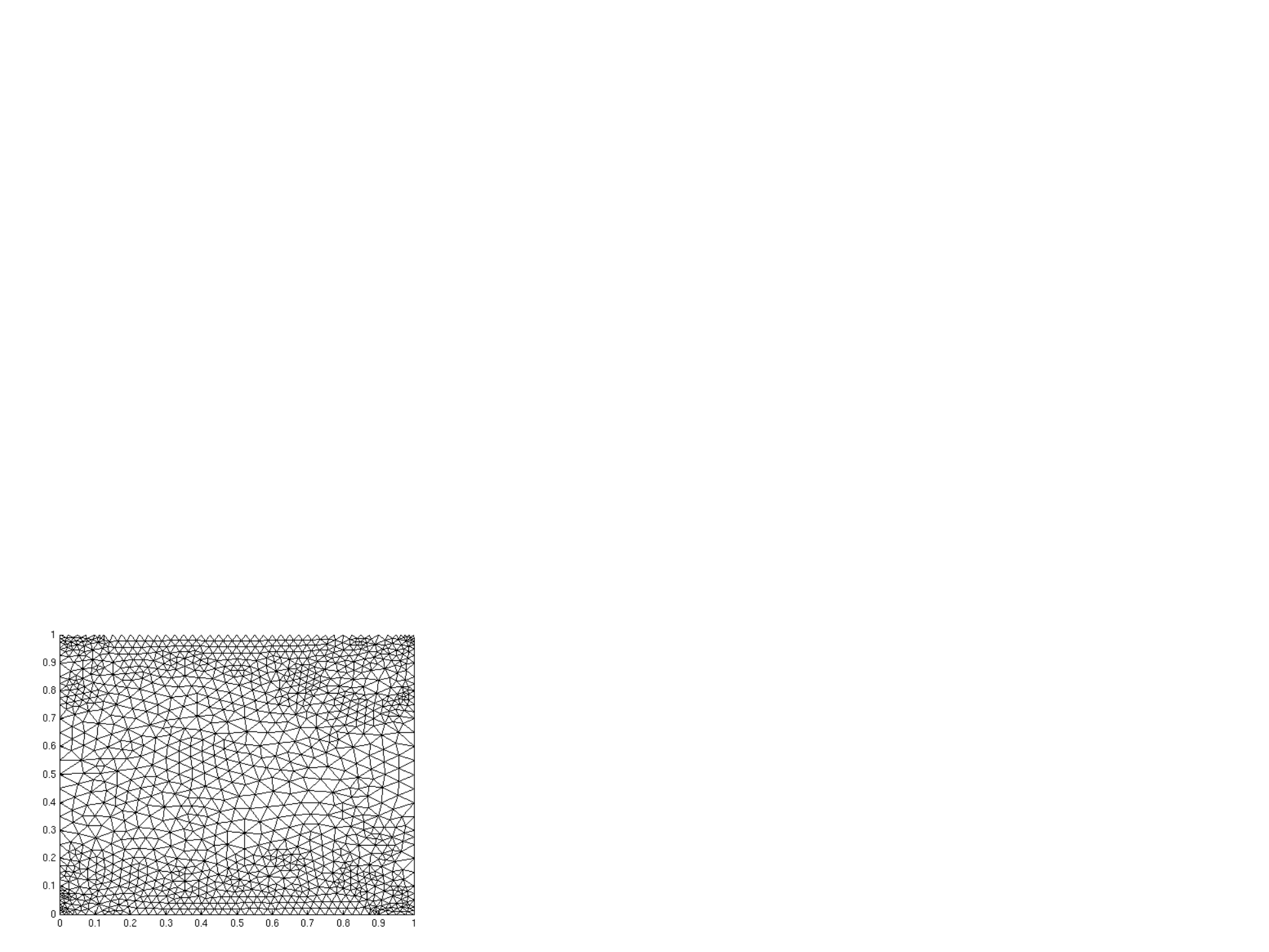}
\caption{Final mesh after adaptive refinement, $t=0.2$}
\label{fig:poismesh02}
\centering
\includegraphics[scale=0.8]{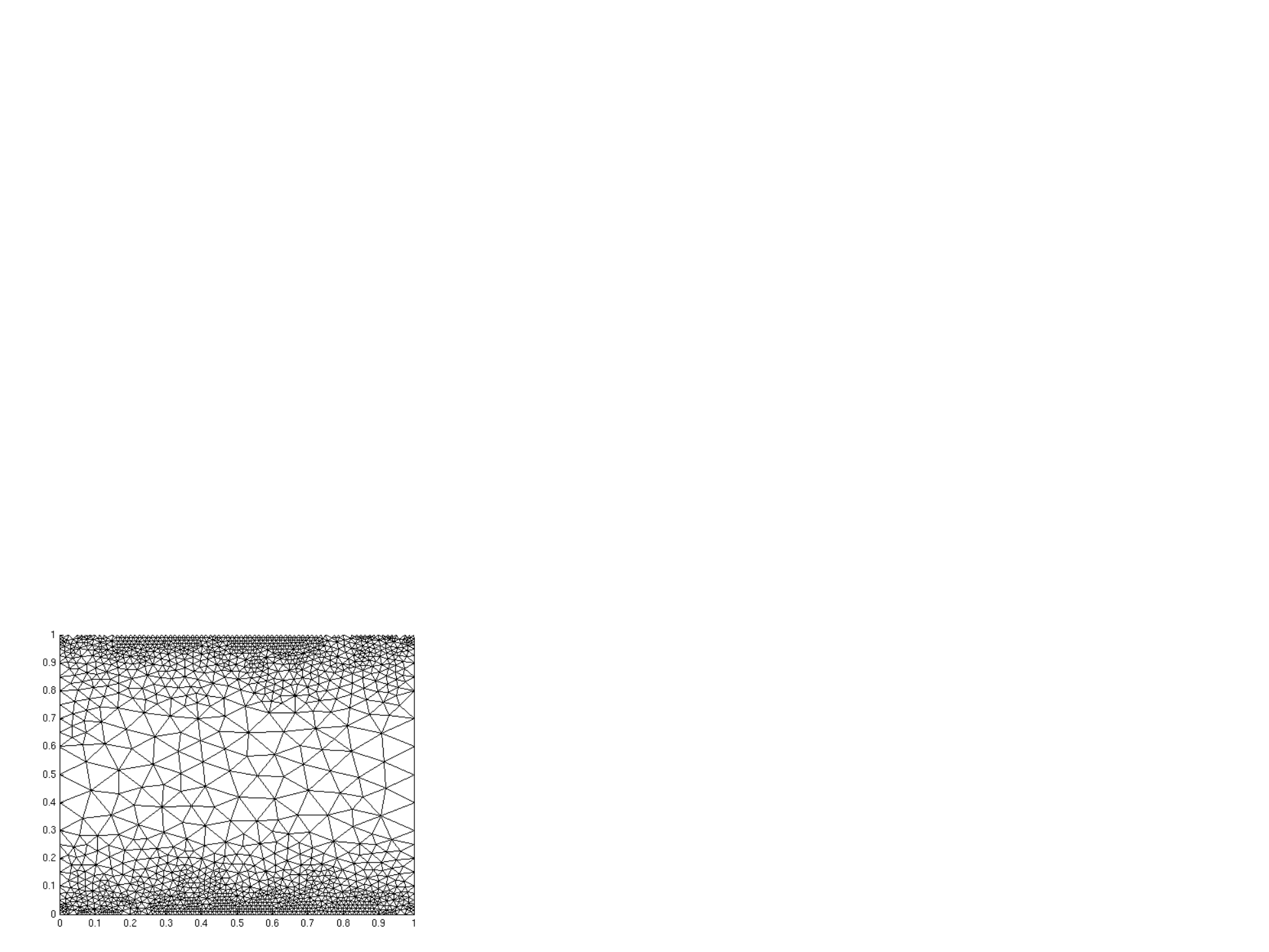}
\caption{Final mesh after adaptive refinement, $t=0.1$}
\label{fig:poismesh01}
\end{minipage}
\hspace{0.5cm}
\begin{minipage}[b]{0.45\linewidth}
\centering
\includegraphics[scale=0.8]{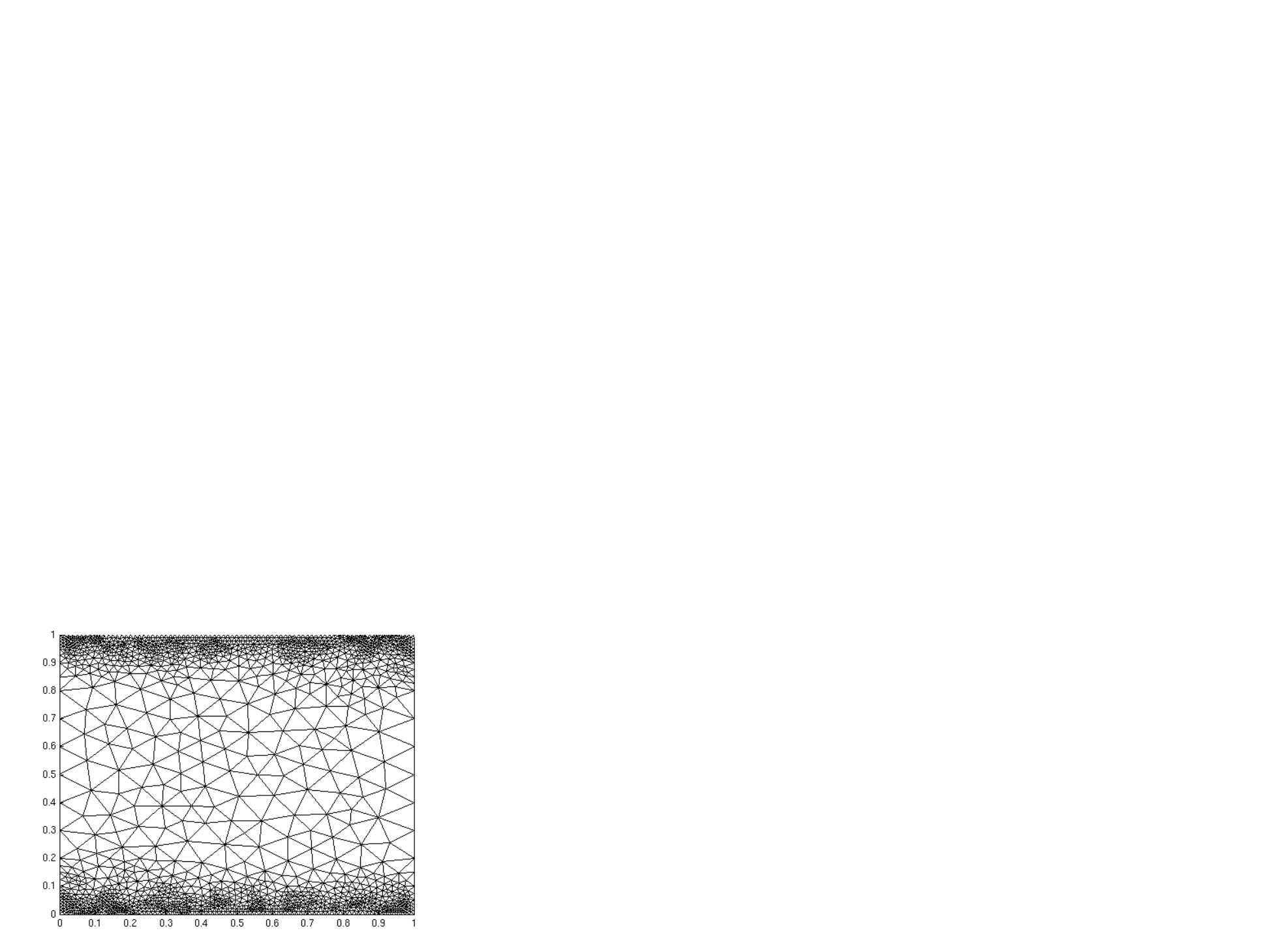}
\caption{Final mesh after adaptive refinement, $t=0.05$}
\label{fig:poismesh005}
\centering
\includegraphics[scale=0.8]{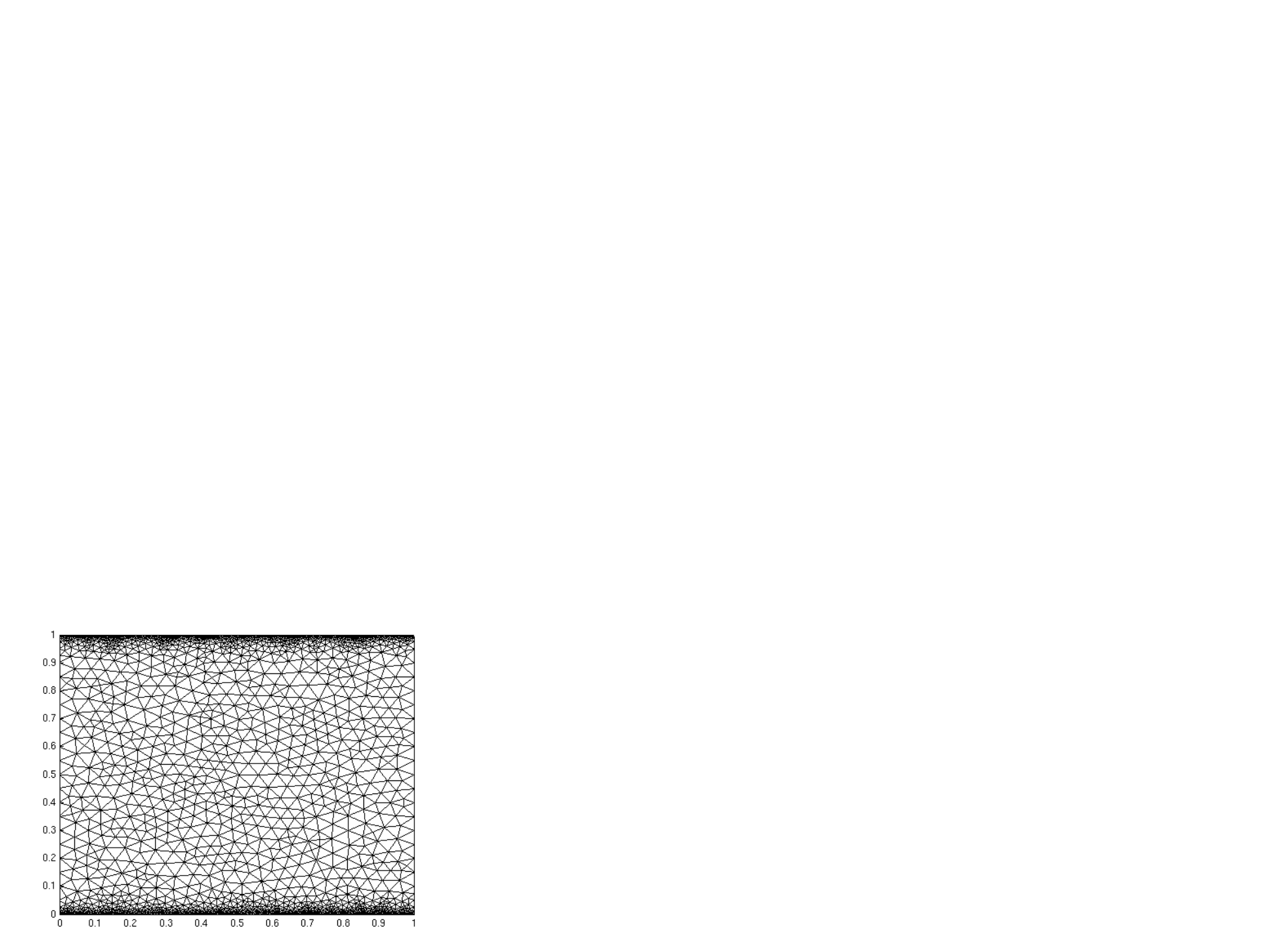}
\caption{Final mesh after adaptive refinement, $t=0.005$}
\label{fig:poismesh0005}
\centering
\includegraphics[scale=0.62]{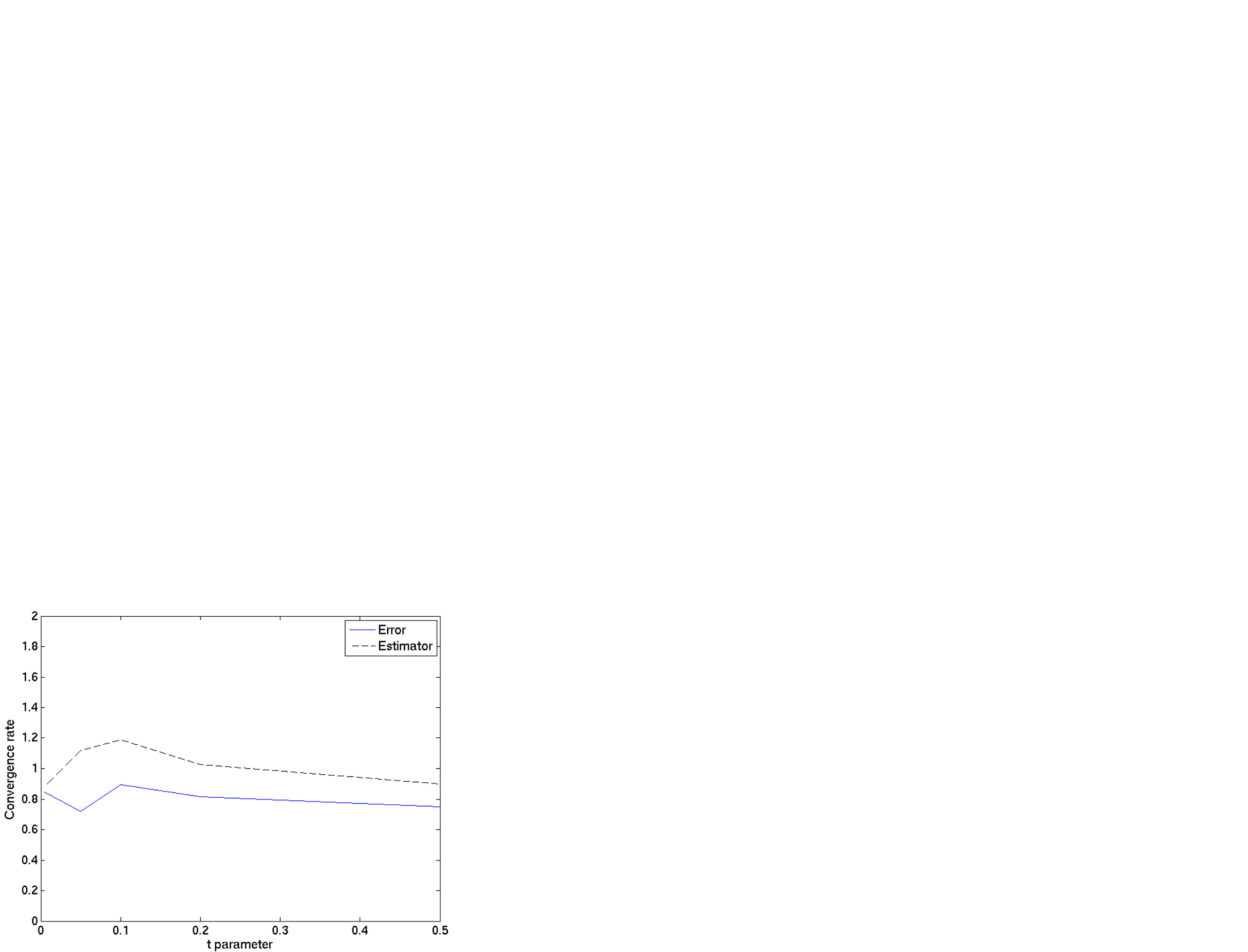}
\caption{Convergence rates of the adaptive solution for different values of $t$.}
\label{fig:rates_poisseuille}
\end{minipage}
\end{figure}

\begin{figure}[!ht]
\begin{minipage}[b]{0.45\linewidth}
\centering
\includegraphics[scale=1.0]{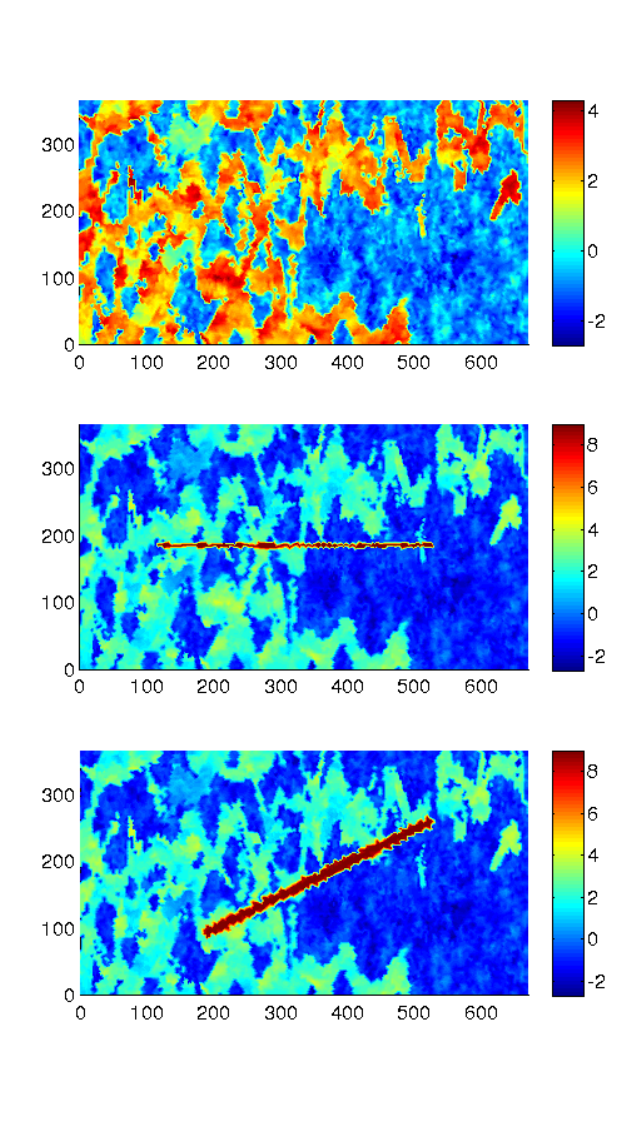}
\caption{Original and modified permeabilities on a logarithmic scale in mD for layer 68 of SPE10.}
\label{fig:perms}
\end{minipage}
\hspace{0.5cm}
\begin{minipage}[b]{0.45\linewidth}
\centering
\includegraphics[scale=1.0]{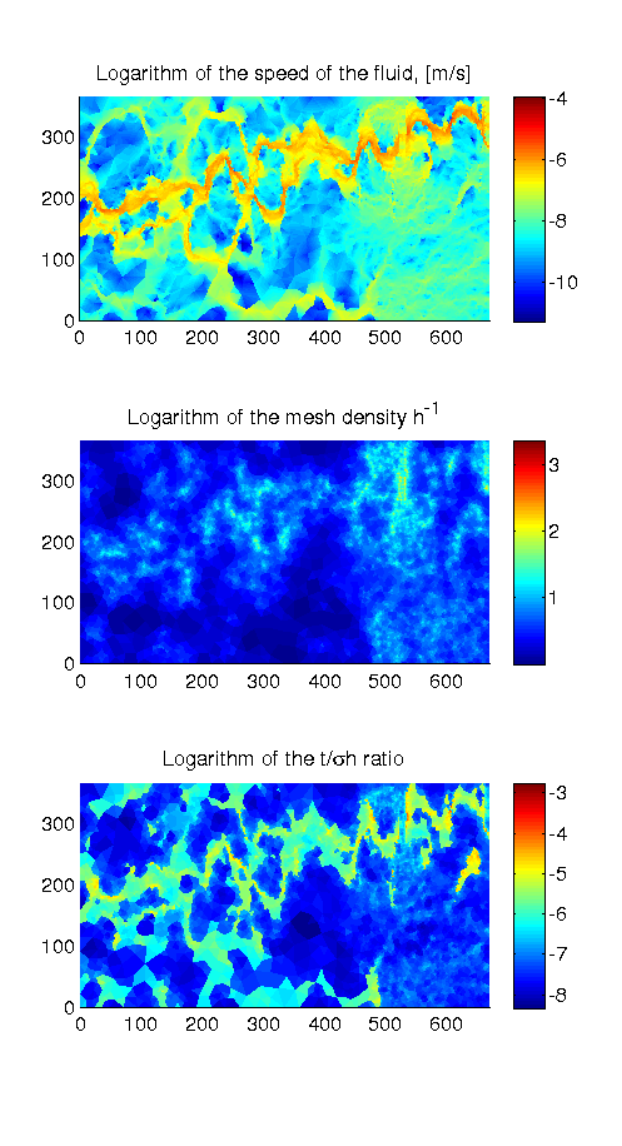}
\caption{Flow in the domain with no streak. Original data from SPE10 dataset, layer 68.}
\label{fig:flow_nostreak}
\end{minipage}
\end{figure}

\clearpage

\begin{figure}[!ht]
\begin{minipage}[b]{0.45\linewidth}
\centering
\includegraphics[scale=1.0]{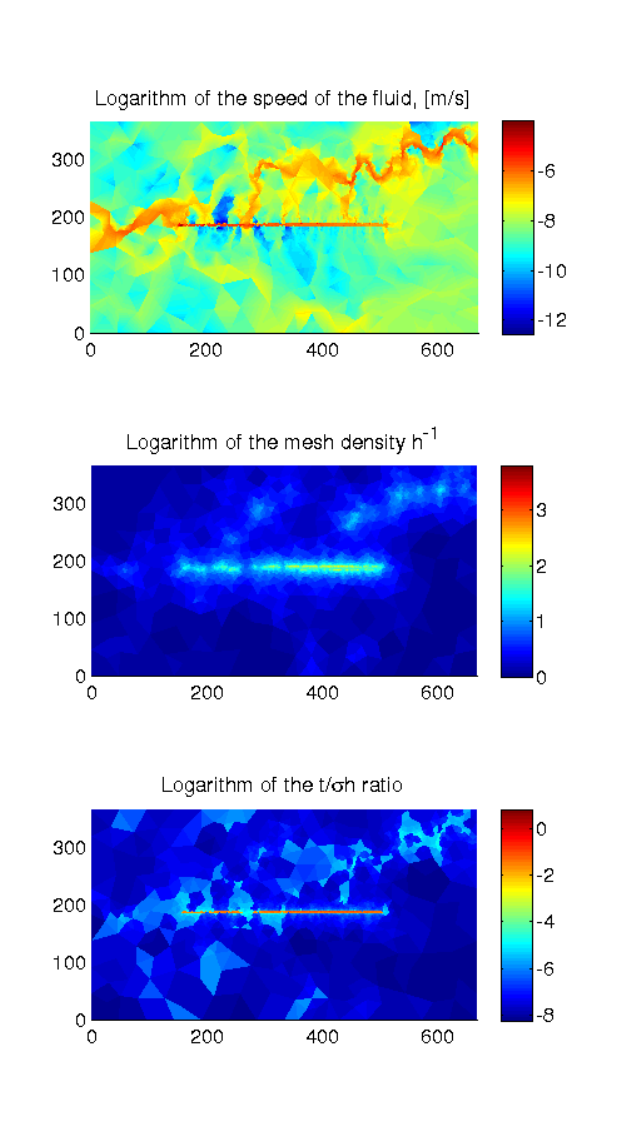}
\caption{Flow in the domain with a horizontal streak.}
\label{fig:flow_streak}
\end{minipage}
\hspace{0.5cm}
\begin{minipage}[b]{0.45\linewidth}
\centering
\includegraphics[scale=1.0]{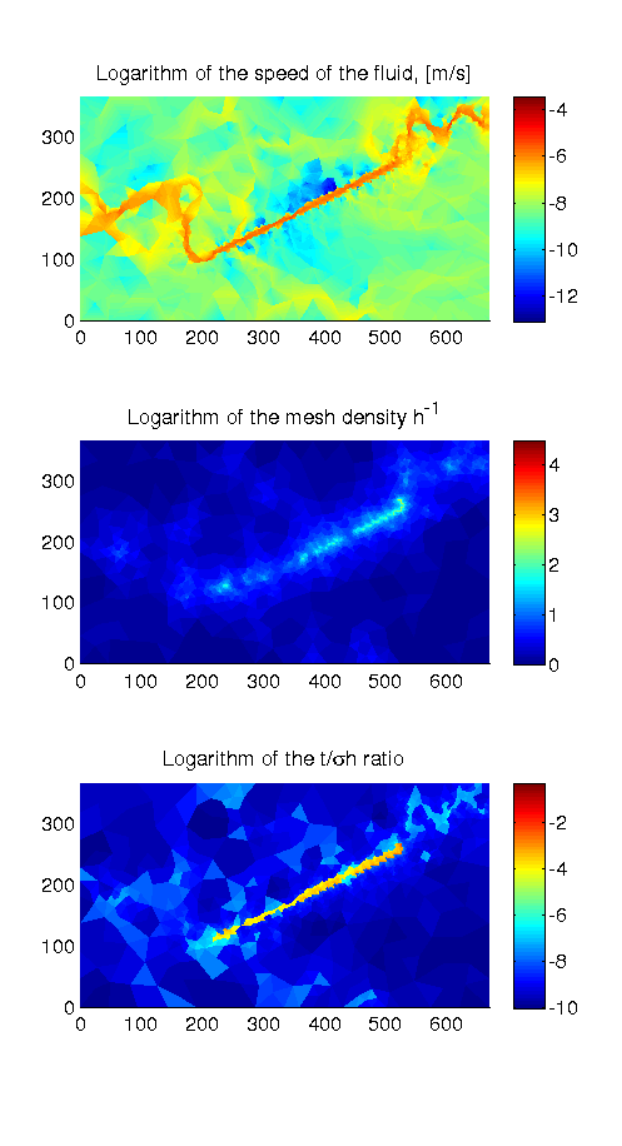}
\caption{Flow in the domain with a tilted streak.}
\label{fig:flow_tilt}
\end{minipage}
\end{figure}

\begin{figure}
\centering
\includegraphics[scale=1]{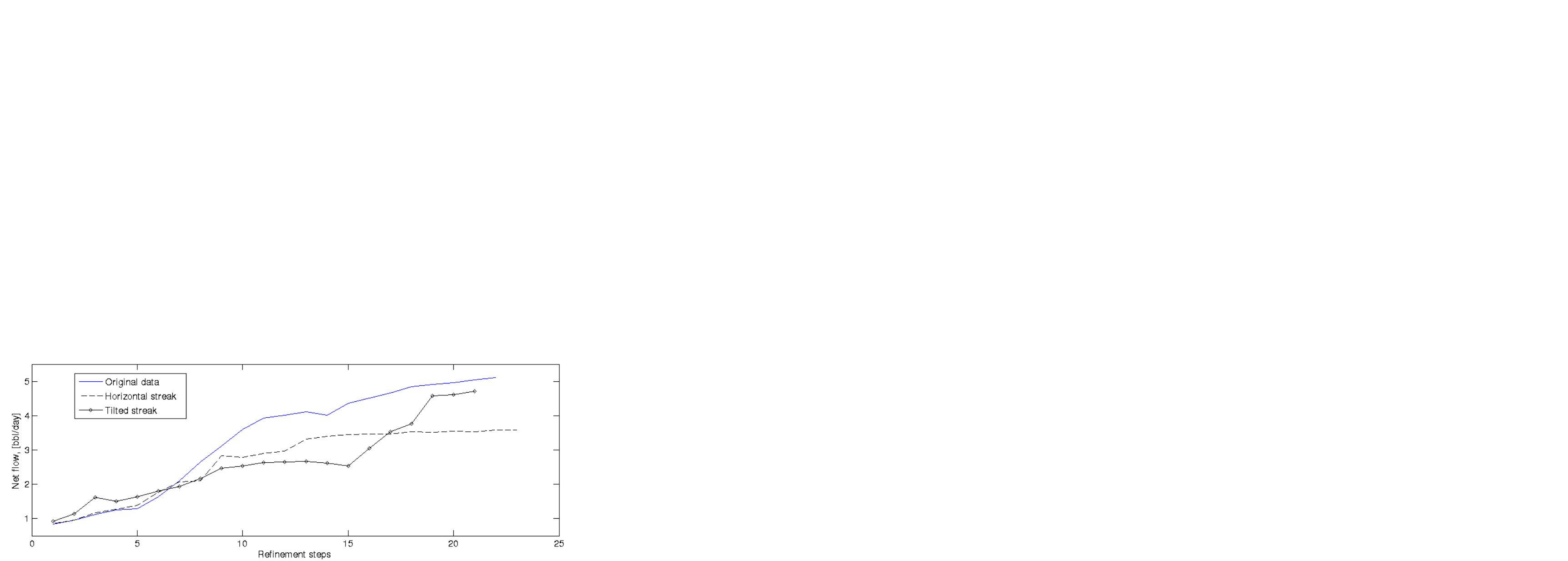}
\caption{Net flow rates for no streak or one internal streak.}
\label{fig:flowrate1}
\end{figure}

\begin{figure}[!hb]
\includegraphics[scale=1.0]{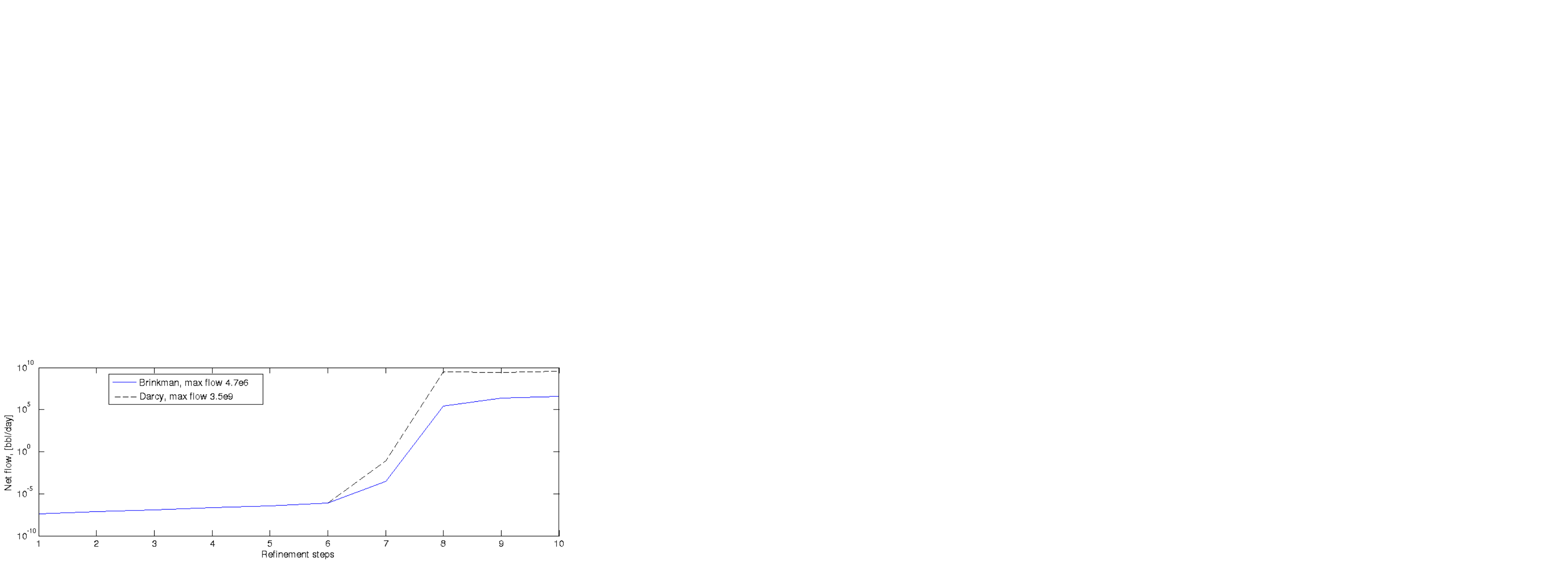}
\caption{Net flow rates for the Brinkman and Darcy models with a piercing crack with a permeability of $10^{15}$ Darcy.}
\label{fig:flowrate2}
\end{figure}

\begin{figure}[!ht]
\begin{minipage}[b]{0.45\linewidth}
\centering
\includegraphics[scale=1.0]{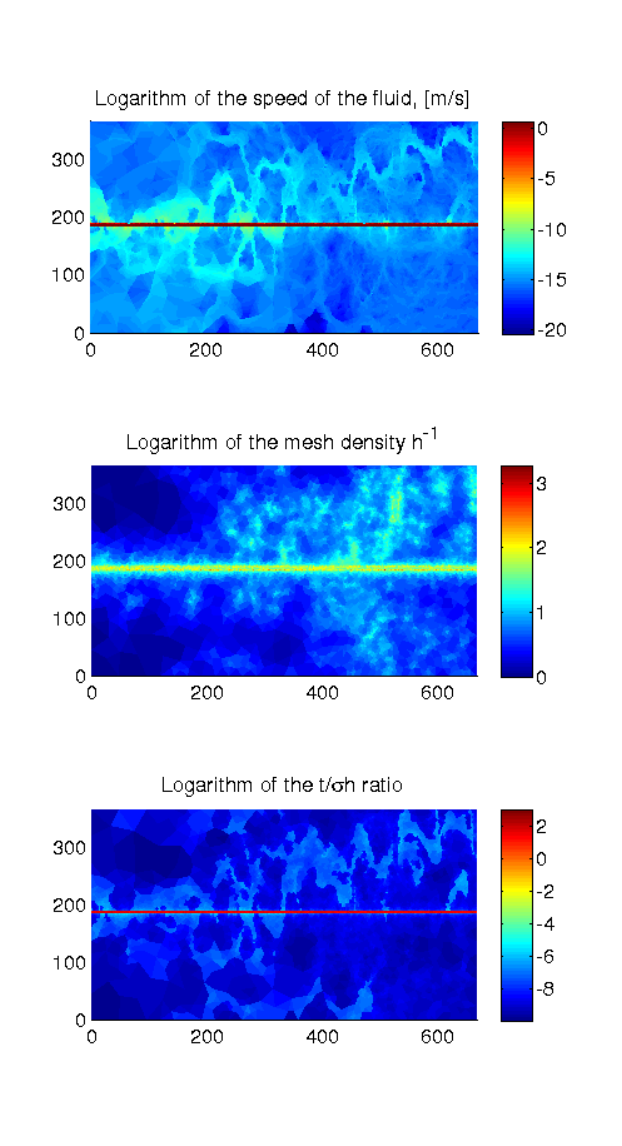}
\caption{Flow in the domain with a piercing crack, Brinkman model.}
\label{fig:piercing_bman}
\end{minipage}
\hspace{0.5cm}
\begin{minipage}[b]{0.45\linewidth}
\centering
\includegraphics[scale=1.0]{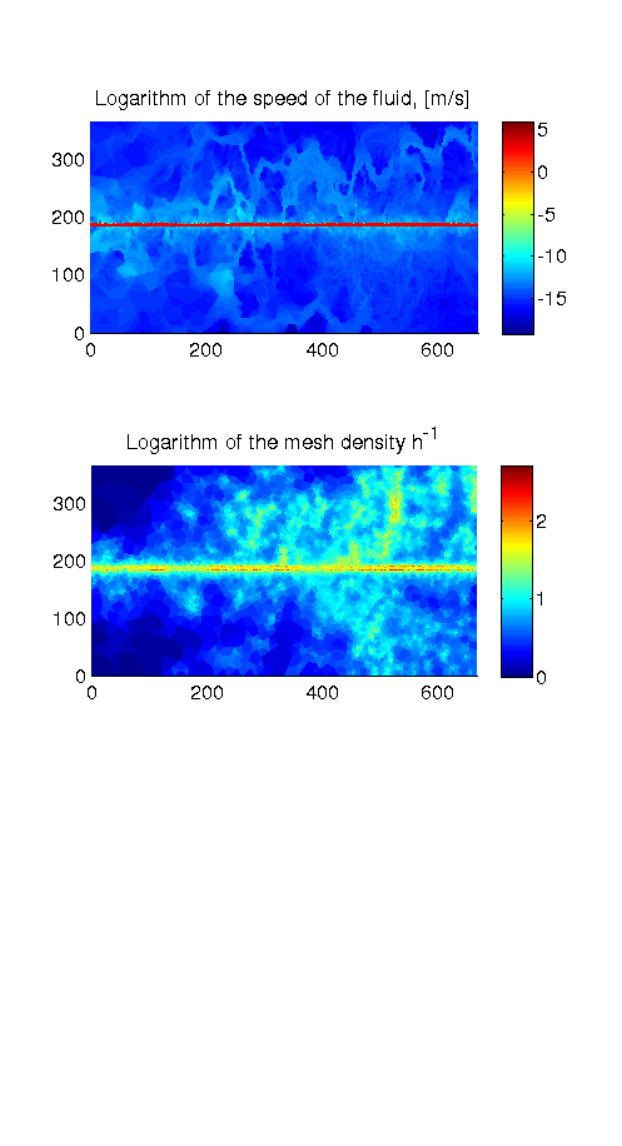}
\caption{Flow in the domain with a piercing crack, Darcy model.}
\label{fig:piercing_darcy}
\end{minipage}
\end{figure}

\end{document}